\documentclass[twoside,11pt,reqno]{amsart}

\usepackage[T1]{fontenc}

\DeclareSymbolFont{symbols}{OMS}{cmsy}{m}{n}

\usepackage{fullpage}

\newcommand{\commentout}[1]{}

\usepackage[utf8]{inputenc}
\usepackage[T1]{fontenc}

\usepackage{mathrsfs}
\usepackage{mathtools}
\usepackage{amsfonts}
\usepackage{amsmath,amssymb,amsthm}
\numberwithin{equation}{section} 
\usepackage{comment}

\usepackage{enumerate}
\usepackage{graphicx,tikz}
\usepackage{ifthen}
\usetikzlibrary{arrows,scopes}
\usetikzlibrary{calc,cd}

\usepackage{setspace}
\usepackage[hidelinks]{hyperref} 
\usepackage[alphabetic,initials]{amsrefs} 

\newtheorem{thm}{Theorem}[section]

\newtheorem{lem}[thm]{Lemma}
\newtheorem{cor}[thm]{Corollary}
\newtheorem{prop}[thm]{Proposition}
\theoremstyle{definition}
\newtheorem{defi}[thm]{Definition}
\newtheorem{notation}[thm]{Notation}

\theoremstyle{remark}
\newtheorem{rmk}[thm]{Remark}

\newcommand{\Hil}{\mathcal{H}}

\newcommand{\slot}{{~\cdot~}}

\newcommand{\C}{\mathcal{C}}

\newcommand{\cF}{\mathcal{F}}

\newcommand{\cE}{\mathcal{E}}

\newcommand{\cP}{\mathcal{P}}
\renewcommand{\P}{\mathcal{P}}
\newcommand{\cC}{\mathcal{C}}
\newcommand{\cS}{\mathcal{S}}

\newcommand{\A}{\mathcal{A}}
\newcommand{\B}{\mathcal{B}}

\newcommand{\M}{\mathcal{M}}

\newcommand{\N}{\mathcal{N}}

\newcommand{\RR}{\mathbb{R}}

\newcommand{\CC}{\mathbb{C}}

\DeclareMathOperator{\End}{End}

\DeclareMathOperator{\Rep}{Rep}

\DeclareMathOperator{\Hom}{Hom}
\DeclareMathOperator{\Ind}{Ind}

\DeclareMathOperator{\Ad}{Ad}

\DeclareMathOperator{\id}{id}
\DeclareMathOperator{\Tr}{Tr}

\DeclareMathOperator{\supp}{supp}

\newcommand{\lqq}{\lq\lq}
\newcommand{\rqq}{\rq\rq\@\xspace}

\newcommand{\dd}{\,\mathrm{d}}

\usepackage{xspace}
\DeclareRobustCommand{\eg}{e.g.\@\xspace}

\DeclareRobustCommand{\cf}{cf.\@\xspace}
\DeclareRobustCommand{\Cf}{Cf.\@\xspace}
\DeclareRobustCommand{\ie}{i.e.\@\xspace}

\DeclareRobustCommand{\eq}{eq.\@\xspace}

\DeclareRobustCommand{\Sec}{Sec.\@\xspace}

\DeclareRobustCommand{\Prop}{Prop.\@\xspace}
\DeclareRobustCommand{\Lem}{Lem.\@\xspace}
\DeclareRobustCommand{\Cor}{Cor.\@\xspace}
\DeclareRobustCommand{\Thm}{Thm.\@\xspace}

\DeclareRobustCommand{\Def}{Def.\@\xspace}
\DeclareRobustCommand{\Rmk}{Rmk.\@\xspace}

\makeatletter
\DeclareRobustCommand{\etc}{%
    \@ifnextchar{.}%
        {etc}%
        {etc.\@\xspace}%
}
\makeatother

\newcommand{\Cstar}{$C^\ast$\@\xspace}

\def\u1net{{\A_\RR}}

\DeclareMathOperator*{\Span}{Span}

\DeclareMathOperator*{\rank}{rank}

\DeclareMathOperator{\Trig}{Trig}

\DeclareMathOperator*{\UCP}{UCP}
\DeclareMathOperator{\Extr}{Extr}
\def\III{{I\!I\!I}}
\def\II{{I\!I}}

\renewcommand{\slot}{\,\cdot\,}

\newcommand{\Cred}{{C^\ast_\mathrm{red}}}

\newcommand{\iotabar}{{\bar\iota}}

\newcommand{\rhobar}{{\bar\rho}}
\newcommand{\sigmabar}{{\bar\sigma}}

\newcommand{\rbar}{{\bar r}}

\makeatletter
\def\subsection{\@startsection{subsection}{2}%
  \z@{.6\linespacing\@plus.7\linespacing}{.5\linespacing}%
  {\normalfont\bfseries}}
\makeatother

\mathtoolsset{showonlyrefs=true} 

\begin{document}

\date{} 
\dateposted{} 

\title{Galois Correspondence and Fourier Analysis on Local Discrete Subfactors}

\address{Department of Mathematics, Morton Hall 321, 1 Ohio University, Athens, OH 45701, USA}
\author{Marcel Bischoff}
\email{bischoff@ohio.edu}
\address{Dipartimento di Matematica, Universit\`a di Roma Tor Vergata, Via della Ricerca Scientifica, 1, I-00133 Roma, Italy}
\author{Simone Del Vecchio}
\email{delvecchio@mat.uniroma2.it}
\address{Dipartimento di Matematica, Universit\`a di Roma Tor Vergata, Via della Ricerca Scientifica, 1, I-00133 Roma, Italy 
\emph{and} Department of Mathematics, Vanderbilt University, 1326 Stevenson Center, Nashville, TN 37240, USA}
\author{Luca Giorgetti}
\email{luca.giorgetti@vanderbilt.edu}
\thanks{M.B.\ is supported by NSF DMS grant 1700192/1821162 \emph{Quantum Symmetries and Conformal Nets}. S.D.V.\ is supported by MIUR Excellence Department Project awarded to the Department of Mathematics, University of Rome Tor Vergata, CUP E83C18000100006. L.G.\ is supported by the European Union's Horizon 2020 research and innovation programme H2020-MSCA-IF-2017 under Grant Agreement 795151 \emph{Beyond Rationality in Algebraic CFT: mathematical structures and models}.}

\begin{abstract}
Discrete subfactors include a particular class of infinite index subfactors and all finite index ones. A discrete subfactor is called local when it is braided and it fulfills a commutativity condition motivated by the study of inclusion of Quantum Field Theories in the algebraic Haag--Kastler setting. In \cite{BiDeGi2021}, we proved that every irreducible local discrete subfactor arises as the fixed point subfactor under the action of a canonical compact hypergroup. In this work, we prove a Galois correspondence between intermediate von Neumann algebras and closed subhypergroups, and we study the subfactor theoretical Fourier transform in this context. Along the way, we extend the main results concerning $\alpha$-induction and $\sigma$-restriction for braided subfactors previously known in the finite index case.
\end{abstract}

\maketitle
\tableofcontents

\setcounter{tocdepth}{3}


\section{Introduction}

The first surprising result which came out of the theory of subfactors is that the Jones index \cite{Jo1983}, a number which measures the relative size of an infinite-dimensional \lqq continuous\rqq tracial factor (a von Neumann algebra endowed with a non-zero tracial state and whose center consists only of the scalar multiples of the identity) inside another factor of the same type can only take discrete values between 1 and 4, and every value above 4. Another unexpected fact which appeared soon after, and which gives an intuition on the previously mentioned breakthrough, is that all possible inclusions of such factors can be described by some kind of symmetry \lqq group-like\rqq object (finite when the index is finite) of the bigger factor, solely determined by the relative position of the smaller factor.

This point of view has been adopted by Ocneanu \cite{Oc1988}, who introduced an invariant for finite index finite depth $\II_1$ subfactors, which he called paragroup and which he used to give a list (later proven to be a complete list as a consequence of Popa's classification theorem \cite{Po1990}) of all possible subfactors with index less than $4$. An abstract paragroup, see also \cite{EvKa1998Book}, is a generalization of a (finite) group together with its unitary representations, where the underlying sets are replaced by a pair of graphs and the group composition law is replaced by the concatenation of paths. In the subfactor context, the paragroup is designed to describe the collection of higher relative commutants of the subfactor arising from the iterated Jones basic construction. 
The higher relative commutants can also be equivalently described in the language of Popa's standard $\lambda$-lattices \cite{Po1995} and Jones' planar algebras \cite{Jo1999}, or categorically as hom spaces in the $2$-\Cstar-category (with two objects $\N$ and $\M$) of $\M$-$\M$, $\M$-$\N$, $\N$-$\M$ and $\N$-$\N$ bimodules generated by the standard $\M$-$\N$ bimodule ${}_{\M} L^2\M_{\N}$ of the subfactor $\N\subset\M$. Throughout this paper we mainly deal with irreducible subfactors, namely with those having trivial relative commutant $\N'\cap\M = \CC 1$. 

As already mentioned by Ocneanu \cite{Oc1988} in the finite index finite depth setting, the two easiest non-group families of examples of paragroups are given by quantum groups and by quotients of groups by non-normal subgroups. 

The first family corresponds to subfactors with depth 2, namely those such that the 3-steps relative commutant $\N' \cap \M_2$ is a factor, where $\N\subset\M\subset\M_1\subset\M_2$ is the beginning of the Jones tower. More precisely, assuming irreducibility and depth 2, there is a finite dimensional Kac algebra (a Hopf *-algebra) in the finite index case \cite{Lo1994}, \cite{Sz1994}, \cite{Da1996}, or a Woronowicz compact quantum group (in the von Neumann algebraic sense \cite{Va2001}, \cite{KuVa2003}) in the infinite index case (assuming the existence of a normal faithful conditional expectation) \cite{HeOc1989}, \cite{EnNe1996}, acting on $\M$ such that $\N$ is the fixed point subalgebra. For depth 2 subfactors, the intermediate algebras $\P$ sitting in $\N\subset\P\subset\M$ are also known to correspond to \lqq subgroups\rqq of the quantum group associated with $\N\subset\M$, via a Galois-type correspondence \cite{NaTa1960}, \cite{IzLoPo1998}, \cite{To2009}.

In this paper, continuing the analysis of \cite{Bi2016}, \cite{BiDeGi2021}, we consider subfactors which are somehow orthogonal to those with depth 2, and which include the second family of examples of paragroups mentioned above (quotients of groups by non-normal subgroups are in fact double coset hypergroups). These subfactors are called local and they appear naturally in the algebraic formulation of Quantum Field Theory \cite{Ha}. They are orthogonal to depth 2 subfactors in the sense that a subfactor which is both local and depth 2 is necessarily a classical compact group fixed point subfactor. Roughly speaking, a subfactor is local if the tensor \Cstar-category generated by the $\N$-$\N$ bimodule ${}_{\N} L^2\M_{\N}$ is braided and if an additional commutativity constraint involving the Pimsner--Popa bases \cite{Po1995a} and the braiding holds.
Assuming irreducibility and locality, in the finite index case there is a finite hypergroup (in the sense of \cite{SuWi2003}) acting on $\M$ and having $\N$ as the fixed point subalgebra \cite{Bi2016}. 
In the infinite index case (assuming a regularity condition called discreteness in \cite{IzLoPo1998}) the same holds for a compact hypergroup \cite{BiDeGi2021}. The subfactor theoretical hypergroup is easy to define. As a set, it consists of all extreme (in the sense of convex sets) $\N$-bimodular unital completely positive maps from $\M$ to $\M$. By definition, it contains the $\N$-fixing *-automorphisms of $\M$, hence it can be regarded as a collection of \lqq generalized gauge symmetries\rqq of $\N\subset\M$ acting on $\M$ by ucp maps. We denote it by $K(\N\subset\M)$. 
\footnote{In the finite index case, the hypergroup structure of $K(\N\subset\M)$ is completely determined by the 2-steps relative commutants $\N'\cap\M_1$ and $\M'\cap\M_2$ together with the subfactor theoretical Fourier transform, \cf Section \ref{sec:Fourier}.}

The purpose of this paper is twofold. On the one hand, we prove a Galois-type correspondence between the intermediate subalgebras $\P$ sitting in $\N\subset\P\subset\M$ and the closed subhypergroups $H$ of $K(\N\subset\M)$.
On the other hand, we study the subfactor theoretical Fourier transform (mainly in the case of local discrete subfactors), we relate it to the hypergroup theoretical Fourier transform and we prove classical inequalities and uncertainty principles.

In Section \ref{sec:prelim}, we review some basics of subfactor theory with emphasis on irreducible type $\III$ subfactors. 
We also recall the results from \cite{BiDeGi2021} which we need in the following sections. 
In particular in Section \ref{sec:dualitydominated}, assuming that $\N\subset\M$ is discrete and local, we recall the identification of the 2-steps relative commutant $\M' \cap \M_2$ with the abelian von Neumann algebra of essentially bounded functions on $K(\N\subset\M)$ with respect to the Haar measure (Proposition \ref{prop:gammagammaisLinfty}). 
We also recall the identification of the convex set of $\N$-bimodular ucp maps $\M\to\M$, denoted by $\UCP_\N(\M)$, with the probability Radon measures on $K(\N\subset\M)$ (Theorem \ref{thm:dualityishomeo}).

In Section \ref{sec:alphaind}, we extend the definitions of $\alpha$-induction and $\sigma$-restriction, introduced in \cite{LoRe1995}, \cite{BcEv1998-I} based on an idea of Roberts \cite{Ro1976}, \cite{Ro1990}, from finite to infinite index discrete subfactors. For later use, we prove the \lqq main formula\rqq for $\alpha$-induction (Theorem \ref{thm:mainformulaalpha}) and the $\alpha\sigma$-reciprocity theorem (Theorem \ref{thm:alphasigmarecip}) for local discrete subfactors. These results should be compared with those contained in \cite{Xu2005} for subfactors arising from strongly additive pairs of conformal nets.

In Section \ref{sec:intermediate}, we show that if $\N\subset\M$ is discrete and local and $\P$ sits in between $\N\subset\P\subset\M$, then $\P\subset\M$ is also discrete and local (Theorem \ref{thm:intermdiscretelocal}). Note that the intermediate inclusion $\P\subset\M$ is harder to treat than $\N\subset\P$, in the sense that it does not even admit in general (in the absence of discreteness and locality) a normal faithful conditional expectation. See \cite{IzLoPo1998}, \cite{To2009} and references therein.

In Section \ref{sec:Galois}, we show the Galois-type correspondence between intermediate algebras $\P$ and closed subhypergroups $H$ of $K(\N\subset\M)$. Given $H$, the associated $\P$ is given by the $H$-fixed point subalgebra $\M^H$. Given $\P$, the associated $H$ is the set of extreme $\P$-fixing ucp maps $\M\to\M$. The two maps are each other's inverse and $H = K(\P\subset\M)$ (Theorem \ref{thm:galoiscorresp}).

In Section \ref{sec:Fourier}, we study the Fourier transform for local discrete subfactors, possibly with infinite index, and for the associated compact hypergroups. The Fourier transform for subfactors (and for the associated paragroups) has been introduced by Ocneanu \cite{Oc1991} in the finite index finite depth $\II_1$ subfactor setting. Since then, it has been a cornerstone in the analysis of subfactors. More recently, is has been extensively studied for finite index subfactors and planar algebras \cite{JiLiWu2016}, for Kac algebras \cite{LiWu2017} and locally compact quantum groups \cite{JiLiWu2018}, proving a number of inequalities and uncertainty principles which generalize classical results from the Fourier analysis on groups. See \cite{JaJiLiReWu2020} for a concise description of the program. In the type $\III$ setting, the Fourier transform can be naturally defined for infinite index subfactors as well.
It is a linear map running between the 2-steps relative commutants $\M' \cap \M_2$ and $\N'\cap \M_1$. We denote it by $\cF: \M' \cap \M_2 \to \N'\cap \M_1$.

In Section \ref{sec:FourierUCP}, we extend the subfactor theoretical Fourier transform to the complex vector space generated by all $\N$-bimodular ucp maps $\M\to\M$ (Proposition \ref{prop:FhatisF}), denoted by $\Span_\CC(\UCP_\N(\M))$, in which the natural domain of definition of the Fourier transform $\M' \cap \M_2$ embeds. We regard this vector space as a noncommutative analogue of the complex bounded Radon measures associated with the subfactor. The composition and a notion of adjoint in $\Span_\CC(\UCP_\N(\M))$ \cite{AcCe1982} provide natural candidates for a convolution and an adjoint of \lqq noncommutative measures\rqq.

In the subsequent sections, we assume in addition that $\N\subset\M$ is discrete and local. In Section \ref{sec:classicalF}, we identify the Fourier transform on $\N\subset\M$ with the ordinary hypergroup theoretical Fourier transform on $K(\N\subset\M)$. In Section \ref{sec:Finequalities}, we show Parseval's identity (Proposition \ref{prop:Parseval}) and Hausdorff--Young's inequality (Proposition \ref{prop:Hausdorff-Young}). In Section \ref{sec:invconv}, we introduce an additional multiplication (called convolution and denoted by $x\ast y$) and an involution operation (denoted by $x^\sharp$) on the von Neumann algebra $\M'\cap\M_2$. These operations are mapped by the Fourier transform to the ordinary product and adjoint of bounded linear operators, $\cF(x\ast y) = \cF(x)\cF(y)$ and $\cF(x^\sharp) = \cF(x)^*$, but they are not defined by these relations, namely $x\ast y := \cF^{-1}(\cF(x)\cF(y))$ and $x^\sharp := \cF^{-1}(\cF(x)^*)$, as it is usually done for finite index subfactors / planar algebras.
Indeed, the inverse Fourier transform $\cF^{-1}$ is globally defined on $\N'\cap \M_1$ if and only if the index is finite. Instead, they are defined by means of the embedding of $\M'\cap\M_2$ into $\Span_\CC(\UCP_\N(\M))$
which corresponds to the embedding of $L^\infty(K(\N\subset\M),\mu_K)$ into the Radon measures on $K(\N\subset\M)$ given by $f \mapsto f \dd \mu_K$ (Remark \ref{rmk:embeddingiflocal}), where $\mu_K$ is the Haar measure.
In Section \ref{sec:convoineq}, we show Young's inequality for the convolution (Proposition \ref{prop:Young}). In Section \ref{sec:inversionuncertainty}, we show the inversion formula (Proposition \ref{prop:inversion}) and a Donoho--Stark uncertainty principle (Proposition \ref{prop:strongDonohoStarkuncertainty}) for the Fourier transform.

\section{Preliminaries}\label{sec:prelim}

Here we recall some basics of subfactor theory \cite{GoHaJo1989}, \cite{JoSu1997}, \cite{Ko1998}. We shall focus on inclusion of infinite factors, mainly type $\III$, with finite or infinite index.

\subsection{The canonical endomorphism}\label{sec:canendo}

Let $\N\subset\M$ be a subfactor acting on a separable Hilbert space $\Hil$. Denote by $\B(\Hil)$ the set of bounded linear operators on $\Hil$. Throughout this paper we shall mainly be interested in irreducible subfactors, \ie $\N' \cap \M = \CC1$. Here $\N'$ is the commutant of $\N$ in $\B(\Hil)$ and $1$ is the identity operator on $\Hil$ sitting in both $\N$ and $\M$. If $\M$ acts standardly on $\Hil$, \ie if it admits a cyclic and separating vector, and if $\N$ and $\M$ are infinite factors, by a result of Dixmier--Mar\'echal \cite{DiMa1971} there are \emph{jointly} cyclic and separating vectors for $\N$ and $\M$ in $\Hil$. We recall below the definition of Longo's \cite{Lo1987} canonical and dual canonical endomorphism:

\begin{defi}
Let $\xi\in\Hil$ be jointly cyclic and separating for $\N$ and $\M$. Denote by $J_{\N,\xi}$, $J_{\M,\xi}$, or simply $J_{\N}$, $J_{\M}$, the respective modular conjugations. Denote by $j_\N := \Ad J_{\N,\xi}$, $j_\M := \Ad J_{\M,\xi}$ the adjoint actions on $\B(\Hil)$. Let 
$$\gamma(x) := j_\N (j_\M(x)), \quad x\in\M.$$ 
Let also $\theta := \gamma_{\restriction\N}$.
\end{defi}

It is easy to see that $\gamma\in\End(\M)$ and $\theta\in\End(\N)$. They depend on the choice of $\xi$ only up to conjugation with a unitary in $\N$, thus their unitary equivalence class is canonical for the subfactor. $\gamma$ is called the \textbf{canonical endomorphism} and $\theta$ the \textbf{dual canonical endomorphism}.

The canonical endomorphism gives a convenient way of describing the Jones tower/tunnel \cite{Jo1983} in the infinite factor setting.
Let $\M_1 := j_\M (\N')$ be the Jones extension of $\M$ given by $\N$. Then $j_\N (j_\M(\M_1)) = \N$. The chosen vector $\xi$ is cyclic and separating for $\M_1$ as well and $J_{\M_1} = J_\M J_\N J_\M$, as $J_\N = J_{\N'}$, hence $J_\M J_{\M_1} = J_\N J_\M$. Setting $\gamma_1(x) := j_\M (j_{\M_1}(x))$ for every $x\in\M_1$, we have that $\gamma_1 \in \End(\M_1)$, $\gamma = {\gamma_1}_{\restriction\M}$, and
$$\theta(\N) \subset \gamma(\M) \subset \gamma_1(\M_1) = \N \subset \M \subset \M_1$$
is the beginning of the Jones tower/tunnel. Moreover, $\theta(x) = j_{\gamma(\M)} (j_{\N}(x))$ for every $x\in\N$.

In the following we shall often distinguish between $\N$ and its embedded image into $\M$.

\begin{defi}
Let $\iota : \N \to \M$ be the inclusion morphism of $\N$ into $\M$. Denote by $\iotabar: \M \to \N$ the morphism defined by $\iotabar := \iota^{-1} \gamma$. The definition is well posed since $\gamma(\M)$ is contained in $\iota(\N)$.
\end{defi}

With this notation,
$$\gamma = \iota \iotabar, \quad \theta = \iota^{-1} \circ \gamma \circ \iota = \iotabar \iota.$$

The morphism $\iotabar$ is called a conjugate of the inclusion morphism $\iota$ \cite{Lo1990}, \cite[\Sec 2.2]{Lo2018}. Note that $\iotabar$ is \emph{not} a conjugate in the 2-categorical sense of the conjugate equations \cite{LoRo1997}, \cite{GiLo2019}, unless the index of $\N\subset\M$ (to be defined below) is finite. Note also that if $\N = \M$, then $\iotabar = \iota^{-1}$.

\subsection{Conditional expectations}\label{sec:condexp}

We recall the definition of conditional expectation, see \cite{St1997} and references therein. A conditional expectation from $\M$ onto $\N$ is a linear map $E:\M \to \M$ such that $E(\M) \subset \iota(\N)$ and

\begin{itemize}
\item[$(i)$] $E(1) = 1$. (unitality)
\item[$(ii)$] $E(\M_+) \subset \M_+$, where $\M_+$ is the positive cone of $\M$. (positivity)
\item[$(iii)$] $E(\iota(y) x \iota(z)) = \iota(y) E(x) \iota(z)$ for every $x\in\M$ and $y,z\in\N$. ($\N$-bimodularity)
\end{itemize}

It follows that $E^2 = E$ and $E(\M) = \iota(\N)$. Moreover, $\|E\| = 1$, where $\|E\|$ is the bounded linear operator norm of $E$ on $\M$ as a Banach space, $E$ is *-preserving and completely positive. 

We shall also use the notation $E:\M \to \N \subset \M$ and denote by $\cE(\M,\N)$ the set of normal (continuous in the ultraweak operator topology) faithful ($E(x^*x) = 0$ for $x\in\M$ implies $x=0$) conditional expectations from $\M$ onto $\N$. The following terminology is due to \cite{FiIs1999}.

\begin{defi}
A subfactor $\N\subset\M$ is called \textbf{semidiscrete} if $\cE(\M,\N)$ is not empty.
\end{defi}

If $\M$ is a $\II_1$ factor, then every subfactor is semidiscrete. Our motivation for studying semidiscrete subfactors in the type $\III$ setting is given by the analysis of nets of local observables \cite{LoRe1995}.

Let $E\in \cE(\M,\N)$. Choose a unit vector $\Omega\in\Hil$ which is cyclic and separating for $\M$ and such that the associated state is $E$-invariant, \ie $(\Omega,x\Omega) = (\Omega,E(x)\Omega)$ for every $x\in\M$. The Jones projection, defined by $e_\N x \Omega := E(x) \Omega$ for every $x\in\M$, depends only on $E$ and on the positive cone of $\Omega$ \cite[\Lem A]{Ko1989}. If $\Omega$ is chosen in the same positive cone with respect to $\M$ of the jointly cyclic and separating vector $\xi$, then $J_{\M,\Omega} = J_{\M,\xi}$ and 
$$\M_1 = j_\M (\N') = \langle \M, e_\N \rangle.$$

We recall the following crucial representation result for conditional expectations in the infinite factor setting \cite[\Prop 5.1]{Lo1989}. Let $\gamma$ and $\theta$ be the canonical and dual canonical endomorphism.

\begin{prop}\label{prop:Estine}
Every $E \in \cE(\M,\N)$ admits a Connes--Stinespring representation in $\Hil$:
$$E = \iota(w)^* \gamma(\slot)\iota(w)$$ 
where $w\in\N$ is an isometry in $\Hom(\id_\N,\theta) := \{y\in\N : y x = \theta(x) y\, \text{ for every } x\in\N\}$, thus $ww^*\in\Hom(\theta,\theta) = \theta(\N)'\cap \N$, and $\gamma_1^{-1}(ww^*) \in \N' \cap \M_1$ is a Jones projection for $E$.
\end{prop}

\begin{rmk}
Assuming irreducibility of $\N\subset\M$, namely $\N'\cap\M = \CC1$, then $\cE(\M,\N)$ is either empty or it consists of a single element $E$. 
If the factors are infinite, the isometry $w\in\Hom(\id_\N,\theta)$ associated with $E$ is also unique (up to a phase factor).
\end{rmk}

For every $E\in \cE(\M,\N)$, there is an associated operator-valued weight $E^{-1} : \N' \to \M' \subset \N'$ in the sense of \cite{Ha1979I}, \cite{Ha1979II} (a possibly unbounded analogue of a conditional expectation), characterized by Kosaki \cite{Ko1986} using the spatial derivative \cite{Co1980}. $E^{-1}$ is normal faithful and semifinite. It is basically never unital, unless $\N = \M$. 

\begin{defi}
$E$ is said to have \textbf{finite index} if $E^{-1}$ is finite (bounded and everywhere defined). The subfactor $\N\subset\M$ has finite index if for some (hence for all) $E\in\cE(\M,\N)$, $E^{-1}$ is finite. 
\end{defi}

Its value on the identity operator is denoted by $\Ind(E) := E^{-1}(1)$ \cite[\Thm 2.2]{Ko1986}. Moreover, $\Ind(E) = \lambda 1$ with $\lambda \in [1,\infty]$ and with the same quantization behaviour below the value $4$ as the Jones index with respect to the trace \cite[\Thm 5.4]{Ko1986}, \cite{Jo1983}.

Let $\widehat E := j_\M E^{-1} j_\M : \M_1 \to \M \subset \M_1$ be the operator-valued weight dual to $E$, in general only normal faithful and semifinite. Discreteness \cite{IzLoPo1998}, which is equivalent to Popa's quasi-regularity \cite{Po1999} \eg when $\N$ is type $\II_1$ \cite[\Prop 3.22]{JoPe2019}, amounts to a further regularity condition on $\widehat E$. This regularity condition is always fulfilled when the index is finite.

\begin{defi}
A subfactor $\N\subset\M$ is called \textbf{discrete} if it is semidiscrete and if some (hence all) $E\in\cE(\M,\N)$ are such that the restriction of $\widehat E$ to $\N' \cap \M_1$ is semifinite.
\end{defi}

The following characterization of discreteness is a consequence of \cite[\Prop 5.2, 5.5]{DeGi2018}, \cf \cite[\Prop 2.5]{BiDeGi2021}. We recall it in the special case of irreducible subfactors.

\begin{prop}
Let $\N\subset \M$ be an irreducible semidiscrete subfactor, with $\N$, $\M$ infinite factors. Then $\N\subset \M$ is discrete if and only if there is a family $\{\psi_i\}_i \subset \M$ fulfilling the following two conditions:
\begin{itemize}
\item[$(i)$] $\psi_i^* e_\N \psi_i$ are non-zero mutually orthogonal projections, $\sum_i \psi_i^* e_\N \psi_i = 1$ in the strong operator topology, and $E(\psi_i \psi_i^*) = 1$. (Pimsner--Popa basis condition)
\item[$(ii)$] $\psi_i \in \Hom(\iota,\iota\rho_i) := \{y\in\M : y \iota(x) = \iota (\rho_i(x)) y\, \text{ for every } x\in\N\}$, where $\rho_i\in\End(\N)$. (charged fields condition)
\end{itemize}
\end{prop}

\begin{rmk}
Condition $(i)$ means that $\{\psi_i\}_i$ is a \textbf{Pimsner--Popa basis} for $\N\subset\M$ with respect to the unique expectation $E$ \cite{PiPo1986},  \cite{Po1995a}. The condition $E(\psi_i \psi_i^*) = 1$ is in general not included in the definition of Pimsner--Popa basis. It guarantees the uniqueness of the Pimsner--Popa expansion. Condition $(ii)$ is an intertwining condition. The terminology \textbf{charged field} comes from the analysis of DHR (after Doplicher--Haag--Roberts) superselection sectors in algebraic Quantum Field Theory \cite{DoHaRo1971}, \cite{DoRo1972}.
\end{rmk}

Note that each $\rho_i$ is a subendomorphism of $\theta$, in symbols $\rho_i \prec \theta$, as $w_i := \iotabar(\psi_i^*) w$ is an isometry in $\Hom(\rho_i,\theta) := \{y\in\N : y \rho_i(x) = \theta(x) y\, \text{ for every } x\in\N\}$. Moreover, $\theta = \bigoplus_i \rho_i$, namely $\theta = \sum_i w_i \rho_i(\slot) w_i^*$, as $\sum_i w_iw_i^* = 1$.
The $\rho_i$ in the above proposition can be chosen to be irreducible, namely $\Hom(\rho_i, \rho_i) = \CC 1$, and with finite tensor \Cstar-categorical 
dimension \cite{LoRo1997}.

In fact by \cite[\Sec 3]{IzLoPo1998}, assuming discreteness and irreducibility of $\N\subset\M$, every irreducible subendomorphism $\rho \prec \theta$ has finite dimension $d(\rho)$. The multiplicity $n_\rho$ of $\rho$ in $\theta$ (the number of subendomorphisms in a direct sum decomposition of $\theta$ unitarily equivalent to the same $\rho$) is also finite and bounded above by the square of the dimension, $n_\rho \leq d(\rho)^2$. 

\begin{notation}
Let $\alpha,\beta:\N\to\M$ be two unital *-homomorphisms between the von Neumann algebras $\N$ and $\M$. 
Let $\Hom(\alpha,\beta) := \{y\in\M : y \alpha(x) = \beta(x) y\, \text{ for every } x\in\N\}$ be the vector space of intertwiners between $\alpha$ and $\beta$. We shall also write $H_\rho := \Hom(\iota,\iota\rho)$ for the spaces of charged fields associated with $\rho\prec\theta$.
\end{notation}

Of particular importance in this paper are the hom spaces $\Hom(\gamma,\gamma) = \gamma(\M)' \cap \M$ and $\Hom(\theta,\theta) = \theta(\N)' \cap \N$. They are respectively isomorphic (via the canonical endomorphisms) to $\M' \cap \M_2$ and $\N' \cap \M_1$, where $\N\subset\M\subset\M_1\subset\M_2$ is the beginning of the Jones tower.

\subsection{Braided and local subfactors}\label{sec:braidedlocalsubf}

Being braided is additional structure on a subfactor. The study of this structure is motivated for instance by the applications to algebraic Quantum Field Theory \cite{Lo1992}, \cite{Re1995}, \cite{BcEv1998-I}, \cite{Xu1998}, \cite{EvPi2003}, \cite{CaKaLo2010}, \cite{BiKaLoRe2014}, where the braiding is the DHR braiding \cite{DoHaRo1971}, \cite{FrReSc1989}, \cite{GiRe2018}.

Let $\N\subset\M$ be an irreducible discrete subfactor, with $\N$, $\M$ type $\III$. 
Denote by $\C \subset \End(\N)$ the rigid \Cstar-tensor category with finite direct sums and subobjects generated by the irreducible (hence finite dimensional) components of $\theta$. See \eg \cite{EGNO15} for the notion of tensor category and \cite{BiKaLoRe2014-2}, \cite{GiYu2019}, \cite{BiChEvGiPe2020} for the unitary/\Cstar case. A unitary braiding on $\C$ is a family of unitaries $\{\varepsilon_{\rho,\sigma}\in\Hom(\rho\sigma,\sigma\rho)\}_{\rho,\sigma\in\cC}$ which is natural (it fulfills $\varepsilon_{\rho',\sigma'} s \rho(t) = t \sigma(s) \varepsilon_{\rho,\sigma}$ for every $s\in\Hom(\rho,\rho')$, $t\in\Hom(\sigma,\sigma')$) and compatible with the tensor structure (it fulfills the so-called hexagonal diagrams). We also write $\varepsilon^+_{\rho,\sigma} :=\varepsilon_{\rho,\sigma}$ and $\varepsilon^-_{\rho,\sigma} := \varepsilon_{\sigma,\rho}^\ast$ for the braiding and its opposite.

\begin{defi}
The subfactor $\N\subset\M$ is called \textbf{braided} if $\C$ admits a unitary braiding.
\end{defi}

Locality for discrete subfactors \cite[\Def 2.16]{BiDeGi2021}, in the finite index setting also called \emph{chiral locality} \cite{BcEv1998-I}, \cite{BcEvKa1999} or \emph{commutativity} of the associated Q-system \cite[\Def 4.20]{BiKaLoRe2014-2}, amounts to a relation between the given braiding on $\C$ and the algebraic structure of the subfactor. By \cite[\Lem 2.17]{BiDeGi2021}, locality for discrete subfactors can be formulated as follows:

\begin{defi}
The subfactor $\N\subset\M$ is called \textbf{local} if it is braided (with braiding $\{\varepsilon_{\rho,\sigma}\}_{\rho,\sigma\in\C}$) and if
$$\iota(\varepsilon^{\pm}_{\sigma,\rho}) \psi' \psi = \psi \psi'$$
for every $\psi \in H_\rho$, $\psi' \in H_\sigma$ and $\rho,\sigma\prec\theta$ irreducible. One can equivalently choose $\varepsilon^+_{\sigma,\rho}$ or $\varepsilon^-_{\sigma,\rho}$.
\end{defi}

\subsection{Compact hypergroups and their actions}\label{sec:cpthyp}

In \cite{Bi2016}, \cite{BiDeGi2021}, we associated with an irreducible local discrete subfactor $\N\subset\M$ a canonical compact hypergroup $K(\N\subset\M)$ acting on $\M$ by unital completely positive maps. The fixed point subalgebra $\M^K$ coincides with $\N$ \cite[\Thm 4.11]{Bi2016}, \cite[\Thm 5.7]{BiDeGi2021}. 
When the subfactor has in addition depth 2, the hypergroup turns out to be a classical compact group \cite[\Cor 1.2]{Bi2016}, \cite[\Thm 7.5]{BiDeGi2021}. 
Thus one can say that $K(\N\subset\M)$ describes the subfactor by means of its \lqq generalized gauge symmetries\rqq.

We now recall the definition of abstract compact hypergroup adopted in \cite[\Def 3.2]{BiDeGi2021}. For finite sets it boils down to the purely algebraic notion of finite hypergroup \cite{SuWi2003}, \cite[\Def 2.3]{Bi2016}.

\begin{defi}\label{def:abstractcpthyp}
Let $K$ be a compact Hausdorff space. Denote by $P(K)$ the convex space of probability Radon measures on $K$, by $C(K)$ the algebra of continuous functions on $K$ and by $\delta_x$ the normalized Dirac measure concentrated in $x$. $K$ is called a \textbf{compact hypergroup} if it is equipped with a biaffine operation, called convolution:
$$P(K)\times P(K)\to P(K), \quad (\mu,\nu) \mapsto \mu\ast\nu,$$ 
with an involution $K \to K, x \mapsto x^\sharp$, and with an identity element $e\in K$ fulfilling the following:
\begin{itemize}
\item[$(i)$] $P(K)$ is a monoid with involution with respect to $\ast, \,^\sharp, \delta_e$, where the involution is defined on probability measures by $\mu^\sharp(E) := \mu(E^\sharp)$ for every Borel set $E\subset K$.
\item[$(ii)$] The involution $x \mapsto x^\sharp$ is continuous and the map:
$$(x,y)\in K\times K \mapsto \delta_x\ast\delta_y \in P(K)$$
is jointly continuous with respect to the weak* topology on measures.
\item[$(iii)$] There exists a (unique) faithful probability measure $\mu_K$, called a \textbf{Haar measure} on $K$,
such that for every $f,g\in C(K)$ and $y\in K$ it holds
\begin{align}
\int_Kf(y\ast x) g(x)\dd\mu_K(x) &= \int_Kf(x)g(y^\sharp \ast x) \dd\mu_K(x), \\
\int_Kf(x\ast y) g(x)\dd\mu_K(x) &= \int_Kf(x)g(x \ast y^\sharp) \dd\mu_K(x),
\end{align}
where
\begin{align}
f(x\ast y ) := (\delta_x\ast\delta_y)(f) = \int_Kf(z)\dd(\delta_x \ast \delta_y)(z).
\end{align}
\end{itemize}
\end{defi}

\begin{rmk}
A compact hypergroup in the sense of the previous definition is also a locally compact hypergroup in the sense of \cite[\Def 2.1]{KaPoCh2010} with $K$ compact. When $K$ is metrizable, it is a compact quantum hypergroup in the sense of \cite[\Def 4.1]{ChVa1999} with $C(K)$ commutative.
Furthermore, this definition sits in between the widely studied notions of DJS hypergroup (after Dunkl--Jewett--Spector) \cite{BlHe1995} and of hypercomplex system \cite{BeKa1998}.
\end{rmk}

The subfactor theoretical hypergroup \cite[\Def 4.48, \Thm 4.51]{BiDeGi2021} is defined as follows:

\begin{defi}\label{def:subfcpthyp}
Let $\N\subset\M$ be an irreducible local discrete type $\III$ subfactor. Let $\UCP_\N(\M)$ be the convex set of $\N$-bimodular (automatically normal and faithful) unital completely positive maps $\M\to\M$. The subfactor theoretical hypergroup as a set is defined by
$$K(\N\subset\M) := \Extr({\UCP}_{\N}(\M))$$
where $\Extr$ denotes the subset of extreme points.
\end{defi}

The convolution in $K(\N\subset\M)$ corresponds to the composition of ucp maps. 
The involution is given by a notion of adjoint of $\phi\in\UCP_\N(\M)$ with respect to an $E$-invariant state on $\M$:

\begin{defi}\label{def:Omegaadj}
Let $\Omega\in \Hil$ be a cyclic and separating unit vector for $\M$ such that the associated state $\omega = (\Omega , \slot \Omega)$ on $\M$ is $E$-invariant. The \textbf{$\Omega$-adjoint} of $\phi$, denoted by $\phi^\sharp$, is the unique (when it exists) ucp map on $\M$ such that
\begin{align}\label{eq:Omegaadj}
(x \Omega, \phi(y) \Omega) = (\phi^\sharp(x) \Omega, y \Omega)
\end{align}
for every $x, y\in\M$. 
\end{defi}

\begin{lem}\label{lem:uniqueinvolutionlocal}
The involution $\phi \mapsto \phi^\sharp$ does not depend on the choice of $\omega$.
\end{lem}

\begin{proof}
It follows from \cite[\Prop 3.8]{BiDeGi2021}.
\end{proof}

\begin{rmk}
The $\Omega$-adjointability of $\phi\in\UCP_\N(\M)$, \ie the existence of $\phi^\sharp$, is guaranteed by the discreteness and locality of $\N\subset\M$, see \cite[\Lem 4.22]{BiDeGi2021} and \cf also Remark \ref{rmk:Vinrelcomm}.
The independence of $\phi^\sharp$ on $\omega$ for $\N$-bimodular ucp maps can also be checked directly, in full generality and without using the uniqueness of the hypergroup theoretical involution. 
\end{rmk}

The identity in $K(\N\subset\M)$ is the trivial automorphism $\id_\M$ of $\M$. The Haar measure on $K(\N\subset\M)$, denoted by $\mu_E$, corresponds to the unique normal faithful conditional expectation $E$ in $\cE(\M,\N)$.
\smallskip

We also recall the definition of action of an abstract compact hypergroup $K$ on a von Neumann algebra $\M$ given in \cite[\Def 5.1]{BiDeGi2021}. Let $\Omega$ be a cyclic and separating unit vector for $\M$.

\begin{defi}
Let $K$ be a compact hypergroup. Denote by $\UCP^\sharp(\M,\Omega)$ the set of $\Omega$-adjointable ucp maps on $\M$, \cf \cite[\Def 2.5]{BiDeGi2021}. 
An \textbf{action} of $K$ on $\M$ by $\Omega$-adjointable ucp maps is a continuous map:
\begin{align}
\alpha:K\to \Extr({\UCP}^{\sharp}(\M,\Omega))
\end{align}
where ${\UCP}^{\sharp}(\M,\Omega)$ is equipped with the pointwise weak operator topology, such that the lift to probability Radon measures
$\tilde\alpha : P(K) \to {\UCP}^{\sharp}(\M,\Omega)$, defined by
\begin{align}\label{eq:liftaction}
(\tilde{\alpha}(\mu))(x) := \int_K (\alpha(k))(x) \dd\mu(k), \quad \mu\in P(K), x\in\M
\end{align}
where the integral is in the weak sense, is an involutive monoid homomorphism:
\begin{align}
\tilde{\alpha}(\mu_1) \circ \tilde{\alpha}(\mu_2) = \tilde{\alpha}(\mu_1 \ast \mu_2), \quad \tilde{\alpha}(\mu)^\sharp = \tilde{\alpha}(\mu^\sharp), \quad \tilde{\alpha}(\delta_e) = \id.
\end{align}
 \end{defi}

For compact groups, the previous definition boils down to an ordinary action by automorphisms.

\subsection{Duality theorem and dominated $\UCP$ maps}\label{sec:dualitydominated}

The first nontrivial part of \cite[\Thm 4.51]{BiDeGi2021}, which states that $K(\N\subset\M)$ fulfills the requirements of Definition \ref{def:abstractcpthyp}, is to show that the extreme points are closed (hence compact). This follows as a consequence of the duality theorem \cite[\Thm 4.34]{BiDeGi2021}, which implies in particular that $K(\N\subset\M)$ is homeomorphic to the Gelfand spectrum of a commutative unital separable \Cstar-algebra, denoted by $\Cred(\N\subset\M)$ \cite[\Def 4.19]{BiDeGi2021} \footnote{$\Cred(\N\subset\M)$ is obtained as a norm closure of a *-algebra $\Trig(\N\subset\M)$ defined in the type $\III$ setting \cite{BiDeGi2021} analogously to a corner of the Popa--Shlyakhtenko--Vaes generalized tube *-algebra \cite{PoShVa2018} in the type $\II_1$ setting.} and canonically associated with the subfactor. Thus $K(\N\subset\M)$ is compact metrizable and
\begin{align}\label{eq:KisGelfandK}
\Cred(\N\subset\M) \cong C(K(\N\subset\M)).
\end{align}
From now on we shall denote $K(\N\subset\M)$ simply by $K$.

\begin{notation}
Denote by $\B(K)$ the Borel $\sigma$-algebra in $K$. Let $M(K,\B(K))$ be the vector space of complex bounded Radon measures on $K$. In this notation, the probability measures $P(K)$ considered in the previous section can be denoted by $P(K,\B(K))$. Let $L^\infty(K, \mu_E)$ be the algebra of essentially bounded measurable functions on $K$ with respect to the Haar measure. \footnote{In Section \ref{sec:Fourier}, we shall use the notation $P(K)$, $M(K)$, $L^\infty(K)$ respectively for ${\UCP}_\N(\M)$, ${\Span}_\CC({\UCP}_\N(\M))$ and $\Hom(\gamma,\gamma)$. The reason is explained in the remainder of this section.}
\end{notation}

Recall the duality theorem \cite[\Thm 4.34]{BiDeGi2021}, which is the main technical result in \cite{BiDeGi2021} and from which \eqref{eq:KisGelfandK} follows by restricting to the extreme points:

\begin{thm}\label{thm:dualityishomeo}
Let $\N\subset \M$ be an irreducible local discrete type $\III$ subfactor. There is an affine homeomorphism denoted by
\begin{align}
\phi \mapsto \mu_\phi
\end{align}
between $\UCP_\N(\M)$ equipped with the pointwise weak operator topology and the state space of $\Cred(\N\subset \M)$, denoted by $\cS(\Cred(\N\subset \M))$, equipped with the weak$^*$ topology.
\end{thm}

We shall not need in this paper the exact definition of $\mu_\phi$, nor of $\Cred(\N\subset \M)$. We will need instead the following two propositions proven in \cite[\Prop 4.42]{BiDeGi2021} and \cite[\Prop 4.44]{BiDeGi2021}:

\begin{prop}\label{prop:SpanUCP}
The map $\phi\mapsto \mu_\phi$ extends to a linear bijection:
$${\Span}_{\CC}({\UCP}_{\N}(\M))\to (\Cred(\N\subset \M))^*$$ 
onto the continuous dual of $\Cred(\N\subset \M)$.
\end{prop}

By the Riesz--Markov theorem, \begin{align}\label{eq:PsPandMisM}
{\UCP}_\N(\M) \cong P(K,\B(K)), \quad {\Span}_\CC({\UCP}_\N(\M)) \cong M(K,\B(K)).
\end{align}

As we recall in more detail in Section \ref{sec:Fourier}, there is a Radon--Nikodym theorem for completely positive maps \cite{Ar1969}, see also \cite{Pa1973}, \cite{Bi2016}. It implies that every $\phi \in \UCP_\N(\M)$ dominated by $E$ (in the sense that $d E - \phi$ is completely positive for some $d>0$) is of the form
\begin{align}\label{eq:RNphi}
\phi = \iota(w)^* x \gamma(\slot) \iota(w)
\end{align}
for some positive operator $x\in\Hom(\gamma,\gamma) = \gamma(\M)' \cap \M$. Recall that $E = \iota(w)^* \gamma(\slot)\iota(w)$ by Proposition \ref{prop:Estine}.

Thus $\Hom(\gamma,\gamma)$ can be viewed as the algebra of bounded densities associated with $\N$-bimodular ucp maps $\M\to\M$ which are dominated by $E$. The following proposition states that $\Hom(\gamma,\gamma)$ is identified with the von Neumann algebra $L^\infty(K, \mu_E)$ via \eqref{eq:PsPandMisM} and \eqref{eq:RNphi}.

\begin{prop}\label{prop:gammagammaisLinfty}
Let $f \mapsto x_f$ be the map defined on positive functions $f\in L^\infty(K,\mu_E)$ such that $\int_K f \dd\mu_E = 1$ by considering the unique positive operator $x_f\in\Hom(\gamma,\gamma)$ such that $\iota(w)^* x_f \iota(w) = 1$ and 
$$f \dd\mu_E = \mu_{\phi_{x_f}}$$ 
where $\phi_{x_f} := \iota(w)^* x_f \gamma(\slot) \iota(w)$. 

Then $f \mapsto x_f$ extends to a normal *-isomorphism from $L^\infty(K,\mu_E)$ onto $\Hom(\gamma,\gamma)$. In particular, $\Hom(\gamma,\gamma)$ is a commutative von Neumann algebra, isometrically isomorphic to $L^\infty(K,\mu_E)$.

Under this identification, 
$$\mu_E(g) = \int_K g \dd\mu_E = E(x_g)$$
for every $g\in L^\infty(K,\mu_E)$.
\end{prop}

\section{$\alpha$-induction for discrete subfactors}\label{sec:alphaind}

The operations of $\alpha$-induction and $\sigma$-restriction have been introduced in \cite{LoRe1995} and further studied in \cite{BcEv1998-I}, \cite{BcEv1999-II}, \cite{BcEv1999-III}, \cite{BcEvKa1999}, \cite{BcEvKa2000}, \cite{CoDoRo2001}. The idea comes from Roberts' cohomological description of superselection sectors in algebraic Quantum Field Theory \cite{Ro1976}, \cite{Ro1990}. If $\N\subset\M$ is a braided type $\III$ subfactor with finite index, $\alpha$-induction and $\sigma$-restriction provide a way of defining endomorphisms of $\M$ starting from endomorphisms of $\N$ and vice versa. 
We refer to \cite[\Sec 3.3]{BcEvKa1999} for the definitions in the finite index subfactor context. In the possibly infinite index setting, they have been studied in \cite{Xu2005} in the context of subfactors coming from strongly additive pairs of conformal nets, and used in \cite{CaCo2001}, \cite{CaCo2005}, \cite{CaHiKaLoXu15} to investigate structural properties of inclusions of nets of local observables. We refer to \cite[\Sec 2]{CaCo2001Siena} for the definition of $\alpha$-induction using cocycles.

In this section, we define $\alpha$-induction and $\sigma$-restriction for braided discrete subfactors and study their properties and mutual relations in the local case.

\begin{defi}
Let $\N\subset\M$ be a braided discrete type $\III$ subfactor. Let $\C \subset \End(\N)$ and $\{\varepsilon^{\pm}_{\rho,\sigma}\}_{\rho,\sigma\in\C}$ be as in Section \ref{sec:braidedlocalsubf}.
For every $\rho\in\C$, define its \textbf{$\alpha$-induction} by
$$\alpha^{\pm,\,\N\subset\M}_\rho := {\iotabar}^{\,-1} \circ \Ad \varepsilon^{\pm}_{\rho,\theta}\circ \rho\circ \iotabar$$
where $\iotabar$ is the conjugate of the inclusion, $\theta$ is the dual canonical endomorphism and $\varepsilon^{\pm}_{\rho,\theta}$ is defined by $\varepsilon^{\pm}_{\rho,\theta} := \sum_i w_i \varepsilon^{\pm}_{\rho,\rho_i} \rho(w_i)^*$. The sum converges in the strong operator topology 
and $\theta = \sum_i w_i \rho_i(\slot)w_i^*$ is a direct sum decomposition of $\theta$ into irreducible subendomorphisms $\rho_i\in\C$ with $w_i$ isometries.

For every $\rho\in\End(\M)$, define its \textbf{$\sigma$-restriction} by
$$\sigma^{\N\subset\M}_\rho := \iotabar \circ \rho \circ \iota.$$
We shall omit the apices $\N\subset\M$ and simply write $\alpha^{\pm}_\rho$ and $\sigma_\rho$, when no confusion arises. 
\end{defi}

Clearly $\sigma_\rho \in \End(\N)$ and $t\in\Hom(\rho_1,\rho_2)$, $\rho_1,\rho_2\in\End(\M)$, implies $\iotabar(t)\in\Hom(\sigma_{\rho_1},\sigma_{\rho_2})$.
The following properties of $\alpha$-induction are well known in the finite index case \cite{BcEv1998-I}, \cite{BcEvKa1999}. First, note that $\Ad \varepsilon^{\pm}_{\rho,\theta}\circ \rho\circ \iotabar (\M) \subset \iotabar(\M)$, see \cite[\Lem 7.3]{DeFiRo2021}, hence $\alpha^{\pm}_\rho$ are well defined and $\alpha^{\pm}_\rho \in \End(\M)$.

\begin{lem}
Let $\rho\in\C$, then
\begin{itemize}
\item[$(1)$] $\alpha^{\pm}_\rho$ both extend $\rho$, namely $\alpha^{\pm}_\rho \iota = \iota \rho$.
\item[$(2)$] $\rho \mapsto \alpha^{\pm}_\rho$ is functorial, namely $t\in\Hom(\rho_1,\rho_2)$, $\rho_1,\rho_2\in\C$, implies $\iota(t)\in\Hom(\alpha^{\pm}_{\rho_1},\alpha^{\pm}_{\rho_2})$.
\end{itemize}
\end{lem}

\begin{proof}
Observe first that $\varepsilon^{\pm}_{\rho,\theta} \in \Hom(\rho\theta,\theta\rho)$. Thus $(1)$ follows as in the finite index case. To show $(2)$, recall that by discreteness of $\N\subset\M$ \cite[\Lem 3.8]{IzLoPo1998}, $\M$ is generated as a von Neumann algebra by $\iota(\N)$ and by the charged fields $\psi'\in H_\tau$ associated with the irreducibles $\tau\prec\theta$. By $(1)$, $\alpha^{\pm}_\rho$ preserves $\iota(\N)$. By naturality of the braiding, $\alpha^{\pm}_\rho(\psi') = \iota(\varepsilon^{\mp}_{\tau,\rho})\psi'$. Thus 
$(2)$ follows by observing in addition that
\begin{align}
\iota(t) \iota(\varepsilon^{\mp}_{\tau,\rho_1})\psi' &= \iota(t \varepsilon^{\mp}_{\tau,\rho_1})\psi' \\
&= \iota(\varepsilon^{\mp}_{\tau,\rho_2}) \iota(\tau(t))\psi' \\
&= \iota(\varepsilon^{\mp}_{\tau,\rho_2}) \psi' \iota(t).
\end{align}
\end{proof}

\begin{rmk}
Other properties of $\alpha$-induction such as $\alpha^{\pm}_{\rho\sigma} = \alpha^{\pm}_\rho \alpha^{\pm}_\sigma$, $\alpha^{\pm}_{\rho\oplus\sigma} = \alpha^{\pm}_\rho \oplus \alpha^{\pm}_\sigma$, $\alpha^{\pm}_{\rhobar} = \overline{\alpha^{\pm}_\rho}$ and $d(\alpha^{\pm}_\rho) = d(\rho)$ now follow as in the finite index case, \cf \cite[\Sec 3]{BcEv1998-I}, \cite[\Sec 4.6]{BiKaLoRe2014-2}. 
\end{rmk}

For the remainder of this section, assume that the subfactor is in addition local.

\begin{lem}\label{lem:alphasuNinM}
Let $\N\subset\M$ be an irreducible local discrete type $\III$ subfactor. For every irreducible $\rho\prec\theta$, the space of charged fields $H_\rho$ coincides with $\Hom(\id_\M, \alpha^{\pm}_\rho)$ and the linear map:
\begin{align*}
\Hom(\alpha^{\pm}_\rho,\id_\M&) \to \Hom(\rho,\theta) \\
t &\mapsto \iotabar(t)w
\end{align*}
is a bijection.
\end{lem}

\begin{proof}
The map sends $\Hom(\alpha^{\pm}_\rho,\id_\M)$ to $\Hom(\rho,\theta)$ and it is injective because $\iotabar(t)w = 0$ implies $\iota(w^*\iotabar(t^*t)w) = E(t^*t) = 0$, hence $t = 0$ by faithfulness of $E$. We only have to show surjectivity. Let $s \in \Hom(\rho,\theta)$. By \cite[\Lem 6.15]{DeGi2018}, which relies only on discreteness, there is a charged field $\psi\in H_\rho$ such that 
$$\iotabar(\psi^*)w = s.$$ 
Since $\psi^* \in H_\rho^* = \Hom(\iota \rho,\iota)$ and $\iota \rho = \alpha^{\pm}_\rho \iota$ by the extension property of $\alpha$-induction, 
$\psi^*$ has the desired intertwining property on $\iota(\N)$. 
By naturality of the braiding, $\alpha^{\pm}_\rho(\psi') = \iota(\varepsilon^{\mp}_{\tau,\rho})\psi'$ for every other $\psi'\in H_\tau$, $\tau\prec\theta$. By locality, we get 
\begin{align}\label{eq:localthenintertalpha}
\psi^* \alpha^{\pm}_\rho({\psi'})^* = \psi^*{\psi'}^*\iota(\varepsilon^{\pm}_{\rho,\tau}) = {\psi'}^*\psi^*.
\end{align}
By discreteness, $\iota(\N)$ and the charged fields $\psi'$ generate $\M$ as a von Neumann algebra, hence we conclude that $\psi^*\in\Hom(\alpha^{\pm}_\rho,\id_\M)$. This shows that $H_\rho = \Hom(\id_\M, \alpha^{\pm}_\rho)$ and surjectivity.
\end{proof}

\begin{rmk}
A version of the equality $H_\rho = \Hom(\id_\M, \alpha^{\pm}_\rho)$ appears also in \cite[\Cor 3.9 (2)]{Xu2005} in the context of DHR endomorphisms and strongly additive pairs of conformal nets \cite[\Sec 3]{Xu2005}.
\end{rmk}

The following is a generalization of the \lqq main formula\rqq for $\alpha$-induction \cite[\Thm 3.9]{BcEv1998-I} to local discrete subfactors. Note that the proof is different from the original one when the index is infinite. The same statement is proven in \cite[\Thm 3.8]{Xu2005} in the context of possibly infinite index subfactors coming from strongly additive pairs of conformal nets.

\begin{thm}\label{thm:mainformulaalpha}
Let $\N\subset\M$ be an irreducible local discrete type $\III$ subfactor. For every $\rho,\sigma \in \C$, the linear map:
\begin{align*}
\Hom(\alpha^{\pm}_\rho,\alpha^{\pm}_\sigma&) \to \Hom(\rho,\theta \sigma) \\
t &\mapsto \iotabar(t)w
\end{align*}
is a bijection.
\end{thm}

\begin{proof}
Observe first that Lemma \ref{lem:alphasuNinM} holds also with $\rho \in \C$ replacing $\rho \prec \theta$ irreducible. Indeed, every $\rho \in \C$ can be written as a finite direct sum of irreducibles $\rho_i$ in $\C$, $i=1,\ldots,n$, \ie $\rho = \sum_i w_i \rho_i(\slot) w_i^*$, where the $w_i$ form a Cuntz algebra of isometries. If $\rho_i$ is not a subendomorphism of $\theta$, for some $i$, \ie $\Hom(\rho_i,\theta) = \{0\}$, then $H_{\rho_i} = \{0\}$, \cf \cite[\Lem 6.15]{DeGi2018}, \cite[\Prop 3.2]{IzLoPo1998}. Indeed, let $\psi\in H_{\rho_i}$, then $\iotabar(\psi^*)w \in\Hom(\rho_i,\theta) = \{0\}$. Hence $E(\psi \psi^*) = \iota(w \iotabar(\psi \psi^*)w) = 0$, which implies $\psi = 0$ by faithfulness of $E$. Moreover, $H_\rho = \bigoplus_i H_{\rho_i}$ and $\Hom(\rho,\theta) = \bigoplus_i \Hom(\rho_i,\theta)$ as vector spaces, namely every $\psi \in H_\rho$ and $v \in \Hom(\rho,\theta)$ can be written uniquely as $\psi = \sum_i \iota(w_i) \psi_i$ and $v = \sum v_i w_i^*$, with $\psi_i \in H_{\rho_i}$ and $v_i \in \Hom(\rho_i, \theta)$.

The analogous of \eqref{eq:localthenintertalpha} holds, namely
\begin{align}
\psi^* \alpha^{\pm}_\rho({\psi'})^* &= \sum_i \psi_i^* \iota(w_i^*) {\psi'}^*\iota(\varepsilon^{\pm}_{\rho,\tau}) \\
&= \sum_i \psi_i^* {\psi'}^* \iota(\tau(w_i^*) \varepsilon^{\pm}_{\rho,\tau}) \\
&= \sum_i \psi_i^* {\psi'}^* \iota(\varepsilon^{\pm}_{\rho_i,\tau} w_i^*) \\
&= \sum_i {\psi'}^* \psi_i^* \iota(w_i^*) = {\psi'}^*\psi^*
\end{align}
for every $\psi'\in H_\tau$, $\tau\prec\theta$, by locality and naturality of the braiding. Thus $\psi \in \Hom(\id_\M, \alpha^{\pm}_\rho)$. Moreover, the map $t \mapsto \iotabar(t)w
$ respects the direct sum decompositions, namely
\begin{align}
\iotabar(\psi^*)w &= \iotabar(\sum_i \psi_i^* \iota(w_i^*))w \\
&= \sum_i \iotabar(\psi_i^*) \theta(w_i^*)w \\
&= \sum_i \iotabar(\psi_i^*) w w_i^*.
\end{align}
The map sends $\Hom(\alpha^{\pm}_\rho,\alpha^{\pm}_\sigma)$ to $\Hom(\rho,\theta\sigma)$. Injectivity follows as in the proof of Lemma \ref{lem:alphasuNinM}. To show surjectivity, we consider the following diagram:
\begin{equation}
\begin{tikzcd}
\Hom(\alpha^{\pm}_{\rho\sigmabar},\id_\M) \arrow[r, " "] \arrow[d, " "]  & \Hom(\rho\sigmabar,\theta) \\
\Hom(\alpha^{\pm}_\rho,\alpha^{\pm}_\sigma) \arrow[r, " "]  & \Hom(\rho, \theta\sigma) \arrow[u, " "].
\end{tikzcd}
\end{equation}
The vertical arrows are the isomorphisms given by Frobenius reciprocity, as $\alpha^{\pm}_{\rho\sigmabar} = \alpha^{\pm}_{\rho} \alpha^{\pm}_{\sigmabar}$ 
and $\alpha^{\pm}_{\sigmabar}$ is a conjugate of $\alpha^{\pm}_{\sigma}$.
Let $s \in \Hom(\rho,\theta\sigma)$ and choose a solution $r_\sigma \in \Hom(\id_\N, \sigmabar\sigma)$, $\rbar_\sigma \in \Hom(\id_\N, \sigma\sigmabar)$ of the conjugate equations for $\sigma$ and $\sigmabar$. Then $\theta(\rbar_\sigma^*) s \in \Hom(\rho\sigmabar,\theta)$ with $\rho\sigmabar\in \C$. By the argument above generalizing Lemma \ref{lem:alphasuNinM}, there is an element $t \in \Hom(\alpha^{\pm}_{\rho\sigmabar}, \id_\M)$ such that $\iotabar(t) w = \theta(\rbar_\sigma^*) s$. By $\alpha^{\pm}_{\rho} \iota = \iota \rho$, we have that $t \iota(\rho(r_\sigma)) \in \Hom(\alpha^{\pm}_{\rho}, \alpha^{\pm}_{\sigma})$. To complete the proof, we show that
\begin{align}
\iotabar(t \iota(\rho(r_\sigma))) w &= \iotabar(t) \theta(\rho(r_\sigma)) w \\
&= \iotabar(t) w \rho(r_\sigma) \\
&= \theta(\rbar_\sigma^*) s \rho(r_\sigma) \\
&= \theta(\rbar_\sigma^* \sigma(r_\sigma)) s = s
\end{align}
by the conjugate equation $\rbar_\sigma^* \sigma(r_\sigma) = 1$, hence the diagram commutes and surjectivity is proven.
\end{proof}

We conclude this section by showing a version of \textbf{$\alpha\sigma$-reciprocity} \cite[\Thm 3.21]{BcEv1998-I}, \cite[\Prop 3.3]{BcEv1999-III} for local discrete subfactors. 

\begin{thm}\label{thm:alphasigmarecip}
Let $\C \subset \End(\N)$ be as in Theorem \ref{thm:mainformulaalpha}. For every $\rho,\sigma \in \C$ and $\beta\in\End(\M)$ such that $\beta \prec \alpha^{\pm}_\sigma$, the linear map:
\begin{align*}
\Hom(\alpha^{\pm}_\rho,\beta&) \to \Hom(\rho,\sigma_\beta) \\
t &\mapsto \iotabar(t)w
\end{align*}
is a bijection.
\end{thm}

\begin{proof}
The map sends $\Hom(\alpha^{\pm}_\rho,\beta)$ to $\Hom(\rho,\sigma_\beta)$ and it is injective as in the proof of Lemma \ref{lem:alphasuNinM}. To show surjectivity, choose an isometry $v\in \Hom(\beta,\alpha^{\pm}_\sigma)$ and observe that $\iotabar(v) \in \Hom(\sigma_\beta, \theta\sigma)$ where we used the fact that $\sigma_{\alpha^{\pm}_\sigma} = \theta\sigma$. Let $s\in\Hom(\rho, \sigma_\beta)$, thus $\iotabar(v)s \in \Hom(\rho,\theta\sigma)$, and consider the following diagram:
\begin{equation}
\begin{tikzcd}
\Hom(\alpha^{\pm}_{\rho},\alpha^{\pm}_{\sigma}) \arrow[r, " "] \arrow[d, " v^* \slot "]  & \Hom(\rho,\theta\sigma) \\
\Hom(\alpha^{\pm}_\rho,\beta) \arrow[r, " "]  & \Hom(\rho, \sigma_\beta) \arrow[u, " \iotabar(v) \slot  "].                                     
\end{tikzcd}
\end{equation}
By Theorem \ref{thm:mainformulaalpha}, there is an element $t\in\Hom(\alpha^{\pm}_{\rho},\alpha^{\pm}_{\sigma})$ such that $\iotabar(t)w = \iotabar(v)s$. Thus $v^*t \in \Hom(\alpha^{\pm}_\rho,\beta)$ and it fulfills $\iotabar(v^*t)w = \iotabar(v^*)\iotabar(v)s = s$, which shows surjectivity.
\end{proof}

\section{Intermediate inclusions}\label{sec:intermediate}

In this section, let $\N\subset\M$ be an irreducible local discrete type $\III$ subfactor and let $\P$ be an intermediate von Neumann algebra, namely $\N\subset\P\subset\M$. We show that $\N\subset\P$ and $\P\subset\M$ are both discrete and local. Note that in general, without assuming locality of $\N\subset\M$, the intermediate inclusion $\P\subset\M$ is not even semidiscrete in general \cite{IzLoPo1998}, \cite{To2009}. 

Denote by $\iota_{\N\subset\M}$, $\iota_{\N\subset\P}$ and $\iota_{\P\subset\M}$ respectively the inclusion morphisms of $\N\subset\M$, $\N\subset\P$ and $\P\subset\M$. Similarly for the conjugate morphism $\iotabar$, for the canonical and dual canonical endomorphism $\gamma$ and $\theta$, for the conditional expectation $E$ and the associated isometry $w$.

\begin{lem}\label{lem:intermsemidiscrete}
The inclusions $\N\subset\P$ and $\P\subset\M$ are semidiscrete, \ie they admit a normal faithful conditional expectation. Moreover, they are both irreducible and $\P$ is a type $\III$ factor.
\end{lem}

\begin{proof}
The restriction of $E_{\N\subset\M}$ to $\P$, denoted by $E_{\N\subset\P} : \P\to\N\subset\P$, is clearly a normal faithful conditional expectation onto $\N$. The existence of a normal faithful conditional expectation $E_{\P\subset\M} : \M \to \P\subset\M$ follows by combining a deep result of Izumi--Longo--Popa \cite[\Thm 3.9]{IzLoPo1998} with a consequence of locality \cite[\Prop 2.19]{BiDeGi2021}. The rest is immediate.
\end{proof}

\begin{rmk}\label{rmk:haarKabsorbshaarH}
The intermediate conditional expectation $E_{\N\subset\P}$ is just the restriction of $E_{\N\subset\M}$ to $\P$.
The other intermediate conditional expectation $E_{\P\subset\M}$ is absorbed by $E_{\N\subset\M}$, namely $E_{\N\subset\M} E_{\P\subset\M} = E_{\N\subset\M}$, because $E_{\N\subset\M} = E_{\N\subset\P} E_{\P\subset\M}$ by uniqueness, as $\N\subset\M$ is irreducible, and $E_{\P\subset\M}^2 = E_{\P\subset\M}$.
\end{rmk}

\begin{lem}\label{lem:NinPisdiscretelocal}
The subfactor $\N\subset\P$ is discrete and local.
\end{lem}

\begin{proof}
It follows from \cite[\Thm 2.7]{To2009}, where it is shown that $\theta_{\N\subset\P} \prec \theta_{\N\subset\M}$. Thus the rigid \Cstar-tensor category generated by the irreducible components of $\theta_{\N\subset\P}$ is a subcategory of $\C$. In particular, it is unitarily braided with the same braiding. By choosing $\psi' \in \Hom(\iota_{\N\subset\P},\iota_{\N\subset\P} \tau)$ for every $\tau\prec\theta_{\N\subset\P}$, we get a complete system of charged fields.
\end{proof}

\begin{prop}\label{prop:alphasigmainterm}
Let $\N\subset\M$ and $\P$ as before.
For every $\rho\prec\theta_{\N\subset\M}$ and $\beta \prec\theta_{\P\subset\M}$, the linear map:
\begin{align}
\Hom(\alpha^{\pm,\N\subset\P}_\rho,\beta&) \to \Hom(\rho,\sigma^{\N\subset\P}_\beta) \\
t &\mapsto \iotabar_{\N\subset\P}(t)w_{\N\subset\P}
\end{align}
is a bijection.
\end{prop}

\begin{proof}
The map $t\mapsto \iotabar_{\N\subset\P}(t)w_{\N\subset\P}$ sends $\Hom(\alpha^{\pm,\N\subset\P}_\rho,\beta)$ to $\Hom(\rho,\sigma^{\N\subset\P}_\beta)$. It is injective because the expectation $E_{\N\subset\P}$ is faithful. To show surjectivity, let $s\in\Hom(\rho,\sigma^{\N\subset\P}_\beta)$ and choose an isometry $v\in\Hom(\beta,\theta_{\P\subset\M})$. Then $\iotabar_{\N\subset\P}(v) s \in \Hom(\rho,\theta_{\N\subset\M})$, where we used the fact that $\sigma^{\N\subset\P}_{\theta_{\P\subset\M}} = \theta_{\N\subset\M}$.
By Lemma \ref{lem:alphasuNinM}, there is a charged field $\psi\in\Hom(\id_\M, \alpha^{\pm,\N\subset\M}_\rho)$ such that $\iotabar_{\N\subset\M}(\psi^*)w_{\N\subset\M} = \iotabar_{\N\subset\P}(v) s$. 

By Lemma \ref{lem:intermsemidiscrete}, the intermediate inclusion $\P\subset\M$ is semidiscrete. Let $w_{\P\subset\M}$ be the unique isometry in $\Hom(\id_\P,\theta_{\P\subset\M})$ associated via the Connes--Stinespring representation with the unique expectation $E_{\P\subset\M}$. To conclude the proof, we show that $\iotabar_{\P\subset\M}(\psi^*)w_{\P\subset\M} \in \Hom(\alpha^{\pm,\N\subset\P}_\rho,\theta_{\P\subset\M})$ and that $v^*\,\iotabar_{\P\subset\M}(\psi^*)\,w_{\P\subset\M} \in \Hom(\alpha^{\pm,\N\subset\P}_\rho,\beta)$ is sent to $s$ by the map $t \mapsto \iotabar_{\N\subset\P}(t)\, w_{\N\subset\P}$. Consider the following diagram:
\begin{equation}
\begin{tikzcd}
\Hom(\alpha^{\pm,\N\subset\M}_\rho,\id_\M) \arrow[rr, " "] \arrow[d, " "] &  & \Hom(\rho,\theta_{\N\subset\M}) \\
\Hom(\alpha^{\pm,\N\subset\P}_\rho,\theta_{\P\subset\M}) \arrow[d, " v^* \slot "]         &  & \\
\Hom(\alpha^{\pm,\N\subset\P}_\rho,\beta) \arrow[rr, " "]         &  & \Hom(\rho,\sigma^{\N\subset\P}_\beta) \arrow[uu, " \iotabar_{\N\subset\P}(v) \slot "].                                         
\end{tikzcd}
\end{equation}
For every $p\in\P$, compute
\begin{align} 
\iotabar_{\P\subset\M}(\psi^*)\, w_{\P\subset\M}\, \alpha^{\pm,\N\subset\P}_\rho(p) &= \iotabar_{\P\subset\M}(\psi^* \iota_{\P\subset\M}(\alpha^{\pm,\N\subset\P}_\rho(p)))\, w_{\P\subset\M} \\
&= \iotabar_{\P\subset\M}(\psi^* \alpha^{\pm,\N\subset\M}_\rho(\iota_{\P\subset\M}(p)))\, w_{\P\subset\M} \\
&= \iotabar_{\P\subset\M}(\iota_{\P\subset\M}(p)\psi^*)\, w_{\P\subset\M} \\
&= \theta_{\P\subset\M}(p)\, \iotabar_{\P\subset\M}(\psi^*)\, w_{\P\subset\M}
\end{align}
provided we show that $\iota_{\P\subset\M} \alpha^{\pm,\N\subset\P}_\rho = \alpha^{\pm,\N\subset\M}_\rho \iota_{\P\subset\M}$. The latter equality is readily proven by first taking $p\in\iota_{\N\subset\P}(\N)$, namely $p =\iota_{\N\subset\P}(n)$ for some $n\in\N$, and computing
\begin{align}
\alpha^{\pm,\N\subset\M}_\rho (\iota_{\P\subset\M}(p)) &= \alpha^{\pm,\N\subset\M}_\rho (\iota_{\N\subset\M}(n)) \\
&= \iota_{\N\subset\M} (\rho(n)) \\
&= \iota_{\P\subset\M} (\iota_{\N\subset\P}(\rho(n))) \\
&= \iota_{\P\subset\M} (\alpha^{\pm,\N\subset\P}_\rho(p)).
\end{align}
Secondly, take $p = \psi' \in \Hom(\iota_{\N\subset\P},\iota_{\N\subset\P}  \tau)$ for $\tau\prec\theta_{\N\subset\P}$ and observe that $\iota_{\P\subset\M}(\psi')$ belongs to $\Hom(\iota_{\N\subset\M},\iota_{\N\subset\M}  \tau)$. 
By naturality of the braiding,
\begin{align}
\alpha^{\pm,\N\subset\M}_\rho(\iota_{\P\subset\M}(\psi')) &= \iota_{\N\subset\M}(\varepsilon^{\mp}_{\tau,\rho})\, \iota_{\P\subset\M}(\psi') \\
&= \iota_{\P\subset\M}(\iota_{\N\subset\P}(\varepsilon^{\mp}_{\tau,\rho}) \psi') \\
&= \iota_{\P\subset\M}(\alpha^{\pm,\N\subset\P}_{\rho}(\psi')).
\end{align}
The subfactor $\N\subset\P$ is discrete \cite[\Thm 2.7]{To2009}, thus $\iota_{\N\subset\P}(\N)$ and the charged fields $\psi'$ generate $\P$ as a von Neumann algebra. Hence $\iota_{\P\subset\M}  \alpha^{\pm,\N\subset\P}_\rho = \alpha^{\pm,\N\subset\M}_\rho  \iota_{\P\subset\M}$ as desired.

To conclude the proof, set $t := v^*\,\iotabar_{\P\subset\M}(\psi^*)\,w_{\P\subset\M} \in \Hom(\alpha^{\pm,\N\subset\P}_\rho,\beta)$ and compute
\begin{align}
\iotabar_{\N\subset\P}(t)\, w_{\N\subset\P} &= \iotabar_{\N\subset\P}(v^*\,\iotabar_{\P\subset\M}(\psi^*)\,w_{\P\subset\M})\, w_{\N\subset\P} \\
&= \iotabar_{\N\subset\P}(v^*)\, \iotabar_{\N\subset\M}(\psi^*)\, \iotabar_{\N\subset\P}(w_{\P\subset\M})\, w_{\N\subset\P} \\
&= \iotabar_{\N\subset\P}(v^*)\, \iotabar_{\N\subset\M}(\psi^*)\, w_{\N\subset\M} \\
&= \iotabar_{\N\subset\P}(v^*)\, \iotabar_{\N\subset\P}(v)\, s = s 
\end{align}
by observing that $\iotabar_{\N\subset\P}(w_{\P\subset\M}) w_{\N\subset\P} \in \N$, that it is an isometry and it belongs to $\Hom(\id_\N, \theta_{\N\subset\M})$, and by uniqueness of the conditional expectation $E_{\N\subset\M}$ and of its associated isometry $w_{\N\subset\M}$. 
\end{proof}

\begin{thm}\label{thm:intermdiscretelocal}
Let $\N\subset\M$ be an irreducible local discrete type $\III$ subfactor. For every intermediate von Neumann algebra $\N\subset\P\subset\M$, the subfactors $\N\subset\P$ and $\P\subset\M$ are type $\III$ irreducible discrete and local.
\end{thm}

\begin{proof}
By Lemma \ref{lem:intermsemidiscrete} and Lemma \ref{lem:NinPisdiscretelocal}, it remains only to show that $\P\subset\M$ is discrete and local.
 
To show discreteness of $\P\subset\M$, observe first that $\theta_{\P\subset\M}$ admits a direct sum decomposition into irreducibles. This is because $\Hom(\theta_{\P\subset\M},\theta_{\P\subset\M}) \cong \P' \cap \M_1^\P$, where $\M_1^\P := J_\M \P' J_\M$ is the Jones extension of $\M$ given by $\P$, and $\P' \cap \M_1^\P \subset \N' \cap \M_1^\N$, where $\M_1^\N := J_\M \N' J_\M$, and $\N' \cap \M_1^\N$ is a direct sum of finite matrix algebras by irreducibility and discreteness of $\N\subset\M$ \cite[\Thm 3.3]{IzLoPo1998}. We have to show that the irreducible components $\beta \prec \theta_{\P\subset\M}$ have finite dimension.
By applying Proposition \ref{prop:alphasigmainterm} to $\beta \prec \theta_{\P\subset\M}$, and $\rho\prec\sigma^{\N\subset\P}_\beta\prec\sigma^{\N\subset\P}_{\theta_{\P\subset\M}}=\theta_{\N\subset\M}$ irreducible, hence $d(\rho)<\infty$, we have that the linear bijection:
$$\Hom(\alpha^{\pm,\N\subset\P}_{\rho},\beta) \to \Hom(\rho,\sigma^{\N\subset\P}_\beta)$$
guarantees the existence of isometries in $\Hom(\beta, \alpha^{\pm,\N\subset\P}_{\rho})$. Thus $\beta \prec \alpha^{\pm,\N\subset\P}_{\rho}$ and 
$$d(\beta) \leq d(\alpha^{\pm,\N\subset\P}_{\rho}) = d(\rho) < \infty.$$ 

To show locality of $\P\subset\M$, take $\psi_1 \in \Hom(\iota_{\P\subset\M}, \iota_{\P\subset\M}  \beta_1)$, $\psi_2 \in \Hom(\iota_{\P\subset\M}, \iota_{\P\subset\M}  \beta_2)$ for $\beta_1,\beta_2\prec\theta_{\P\subset\M}$ irreducible, together with the previously mentioned isometries $t_1\in \Hom(\beta_1,\alpha^{\pm,\N\subset\P}_{\rho_1})$, $t_2\in \Hom(\beta_2,\alpha^{\pm,\N\subset\P}_{\rho_2})$ for $\rho_1\prec\sigma^{\N\subset\P}_{\beta_1}$, $\rho_2\prec\sigma^{\N\subset\P}_{\beta_2}$ irreducible. Recall that both $\rho_1,\rho_2\prec\theta_{\N\subset\M}$. Consider the relative braiding between $\beta_1$ and $\beta_2$ introduced in \cite[\Lem 3.11]{BcEv1999-III}, namely
$$\varepsilon^{\pm,\text{rel}}_{\beta_1,\beta_2} := t_2^*\, \alpha^{\mp,\N\subset\P}_{\rho_2}(t_1^*)\, \iota_{\N\subset\P}(\varepsilon^{\pm}_{\rho_1,\rho_2})\, \alpha^{\pm,\N\subset\P}_{\rho_1}(t_2)\, t_1$$
which is independent of the choice of $\rho_1, \rho_2$ and $t_1,t_2$. The finite index assumption made in \cite{BcEv1999-III} is in fact not needed, only discreteness is needed, \cf \cite[\Sec 2]{BcEvKa2000}.
By Corollary \cite[\Cor 3.13]{BcEv1999-III}, the family $\{\varepsilon^{\pm,\text{rel}}_{\beta_1,\beta_2} \in \Hom(\beta_1\beta_2,\beta_2\beta_1), \beta_1,\beta_2\prec\theta_{\P\subset\M}\}$ extends to a unitary braiding on the rigid \Cstar-tensor category generated by the irreducible components of $\theta_{\P\subset\M}$. We have to show that
$$\iota_{\P\subset\M}(\varepsilon^{\pm,\text{rel}}_{\beta_1,\beta_2})\, \psi_1\psi_2 = \psi_2\psi_1$$
for every $\psi_1$, $\psi_2$ as above. Denote for short $\iota_{\P\subset\M}$ by $\iota$ and compute
\begin{align}
\iota(\varepsilon^{\pm,\text{rel}}_{\beta_1,\beta_2})\, \psi_1\psi_2 &= \iota(t_2^*\, \alpha^{\mp,\N\subset\P}_{\rho_2}(t_1^*)\,\iota_{\N\subset\P}(\varepsilon^{\pm}_{\rho_1,\rho_2})\, \alpha^{\pm,\N\subset\P}_{\rho_1}(t_2)\, t_1)\, \psi_1\psi_2 \\
&= \iota(\beta_2(t_1))^* \iota(t_2)^* \iota_{\N\subset\M}(\varepsilon^{\pm}_{\rho_1,\rho_2})\, \iota(t_1)\psi_1\, \iota(t_2)\psi_2 \\
&= \iota(\beta_2(t_1))^* \iota(t_2)^* \iota(t_2)\psi_2 \, \iota(t_1)\psi_1 \\
&= \psi_2 \psi_1
\end{align}
where we used the fact that $\iota(t_1)\psi_1 \in \Hom(\iota_{\N\subset\M}, \iota_{\N\subset\M}  \rho_1)$, $\iota(t_2)\psi_2 \in \Hom(\iota_{\N\subset\M}, \iota_{\N\subset\M}  \rho_2)$, locality of $\N\subset\M$ and $t_1^*t_1 = 1$, $t_2^*t_2 = 1$. Thus the proof is complete.
\end{proof}

\section{Galois correspondence}\label{sec:Galois}

Combining Theorem \ref{thm:intermdiscretelocal} with \cite[\Thm 4.51]{BiDeGi2021}, we give a Galois-type correspondence between intermediate von Neumann algebras $\N \subset \cP \subset \M$ and closed subhypergroups of $K(\N\subset\M)$ considered in Definition \ref{def:subfcpthyp}. The following definition should be compared with \cite[\Def 1.5.1]{BlHe1995} in the DJS hypergroup setting. 

\begin{defi}
Let $K$ be a compact hypergroup in the sense of Definition \ref{def:abstractcpthyp}.
A \textbf{closed subhypergroup} of $K$ is a closed subset $H\subset K$ which is closed under the operations of $K$, namely
$\delta_x\ast\delta_y \in P(H)$, $x^\sharp \in H$ for every $x,y\in H$, $e\in H$, and which admits a Haar measure in $P(H)$ fulfilling $(iii)$ in Definition \ref{def:abstractcpthyp}.
\end{defi}

Recall that $K(\N\subset \M)$ acts faithfully and minimally on $\M$ \cite[\Def 5.1, \Thm 5.7]{BiDeGi2021} and that the fixed point subalgebra $\M^{K(\N\subset \M)}$ coincides with $\N$.

\begin{thm}\label{thm:galoiscorresp}
Let $\N\subset \M$ be an irreducible local discrete type $\III$ subfactor. Denote $K(\N\subset\M)$ simply by $K$. There is a bijective correspondence between the intermediate von Neumann algebras $\N\subset\P\subset\M$ and the closed subhypergroups $H \subset K$ given by 
$$H \mapsto \M^H, \quad \P \mapsto \{\phi \in K : \phi_{\restriction \iota(\P)} = \id_\P\}$$
such that $H = K(\P\subset\M)$.
\end{thm}

\begin{proof}
Given an intermediate von Neumann algebra $\P$, by Theorem \ref{thm:intermdiscretelocal}, $\P\subset\M$ is discrete and local. Let $H := \{\phi \in K : \phi_{\restriction \iota(\P)} = \id_\P\}$. Then $H = K(\P\subset\M)$. Indeed, every $\phi\in H$ is an extreme point of $\UCP_\P(\M)$, thus $\phi\in K(\P\subset\M)$.
Vice versa, every $\phi\in K(\P\subset\M)$ is extreme in $\UCP(\M)$ by \cite[\Lem 4.49]{BiDeGi2021}, thus in $\UCP_\N(\M)$, and $\phi\in H$.
It follows that $H$ is a closed subhypergroup of $K$ with the same convolution and involution on probability measures and with the same identity element.
Indeed, $\delta_{\phi_1}\ast\delta_{\phi_2} \in P(H)$ for every $\phi_1, \phi_2 \in H$, as the convolution is defined by the composition of ucp maps, and $\phi_1 \circ \phi_2\in\UCP_\P(\M)$ if $\phi_1, \phi_2 \in K(\P\subset\M)$.
Moreover, every $E_{\N\subset\M}$-invariant state on $\M$ is also $E_{\P\subset\M}$-invariant by Remark \ref{rmk:haarKabsorbshaarH}, thus by Lemma \ref{lem:uniqueinvolutionlocal} the involution on $K(\P\subset\M)$ as defined in Definition \ref{def:Omegaadj} agrees with the involution of $K$ restricted to $H$. 

Vice versa, given a closed subhypergroup $H$, by definition it acts faithfully and minimally on $\M$.
Let $\P := \M^H$. Again by Theorem \ref{thm:intermdiscretelocal} and by the uniqueness statement for compact hypergroup actions \cite[\Prop 5.4]{BiDeGi2021}, we conclude that $H = K(\P\subset\M)$ 
\end{proof}

\begin{rmk}
The previous theorem generalizes the Galois correspondence established in \cite[\Prop 4.13]{Bi2016} from finite to infinite index subfactors.

It generalizes also \cite[\Thm 3.15]{IzLoPo1998} from minimal actions of compact groups to hypergroups, as \emph{every} compact group fixed point irreducible subfactor $\M^G \subset \M$ is local with respect to the symmetric braiding coming from $\Rep(G)$ \cite[\Prop 9.3]{BiDeGi2021}.
\end{rmk}

\section{Fourier transform}\label{sec:Fourier}

Let $\xi\in\Hil$ be a jointly cyclic and separating unit vector for $\N$ and $\M$. Denote by $\gamma\in\End(\M)$, $\theta\in\End(\N)$ and $\gamma_1\in\End(\M_1)$ the associated canonical endomorphisms as in Section \ref{sec:canendo}.

Recall that the beginning of the Jones tower/tunnel in the infinite factor setting reads:
$$\theta(\N) \subset \gamma(\M) \subset \gamma_1(\M_1) = \N \subset \M \subset \M_1.$$ 

The subfactor theoretical Fourier transform has been introduced in the finite index type $\II_1$ setting by Ocneanu \cite{Oc1988}, \cite[\Sec II.7]{Oc1991}.  
It can be defined using Jones projections and conditional expectations in the tower \cite[\Def 2.16]{Bi1997}, \cite[\Sec 3]{Sa1997}, or graphically in the language of planar algebras as a map running between $n$-box spaces \cite[\Sec 3]{BiJo2000}. We shall be interested in the case $n=2$. We recall below the description of $2$-box spaces (and in the next section of the Fourier transform) in terms of the canonical endomorphisms.
Note that the Fourier transform in the infinite factor setting is defined naturally for semidiscrete subfactors, not necessarily discrete nor with finite index.

Let $H := {}_\M{L^2\M}_\N$ be the standard $\M$-$\N$ bimodule associated with $\N\subset\M$. Let $\bar H$ be the conjugate $\N$-$\M$ bimodule. Denote by $\boxtimes$ the Connes fusion relative tensor product \cite{Sa1983}. Then $\bar H \boxtimes_\M H \cong {}_\N{L^2\M}_\N \cong {}_\N{{}_\theta L^2 \N}_\N$ in the notation of \cite[\Sec 2.2]{Lo2018} and $H \boxtimes_\N \bar H \cong {}_\M{L^2\M_1}_\M \cong {}_\M{{}_\gamma L^2 \M}_\M$. \Cf \cite[\Prop 3.1, 3.2]{Bi1997} and \cite[\Thm 2.50]{JoPe2011}, \cite[\Thm 5.4]{DaGhGu2014} in the type $\II_1$ setting. The bimodule intertwiner algebras are then identified with the higher (in this case 2-step) relative commutants:
\begin{align}
\Hom_{\N\text{-}\N} (\bar H \boxtimes_\M H, \bar H \boxtimes_\M H) &\cong \Hom(\theta,\theta) \\
\Hom_{\M\text{-}\M} (H \boxtimes_\N \bar H, H \boxtimes_\N \bar H) &\cong \Hom(\gamma,\gamma)
\end{align}
where $\Hom(\gamma,\gamma) = \gamma(\M)' \cap \M$ and $\Hom(\theta,\theta) = \theta(\N)' \cap \N$ by definition.

\subsection{Extension of the Fourier transform to $\UCP$ maps}\label{sec:FourierUCP}

In this section, let $\N\subset\M$ be an irreducible semidiscrete ($\cE(\M,\N) \neq \emptyset$) type $\III$ subfactor. Denote by $E$ the unique element in $\cE(\M,\N)$. Recall the notation $\UCP_\N(\M)$ for the $\N$-bimodular ucp maps $\M\to\M$.

\begin{notation}
For ease of notation, we omit $\iota$ symbols and write either just $\gamma (= \iota\iotabar)$ or $\theta (= \iotabar\iota)$ in place of $\iotabar$ when applied either to elements in $\M$ or in $\iota(\N)$, identified with $\N$.
\end{notation}

\begin{defi}\label{def:subfF}
The subfactor theoretical Fourier transform \footnote{It is graphically described by a $90^{\circ}$-rotation: $x\in\Hom(\iota\iotabar,\iota\iotabar) \mapsto \iotabar\iota(w^*) \iotabar(x) w \in \Hom(\iotabar\iota,\iotabar\iota)$ as $w\in\Hom(\id_\N,\iotabar\iota)$.} is defined by 
\begin{align}
\cF : \Hom(\gamma,&\gamma) \to \Hom(\theta,\theta)\\
x &\mapsto \theta(w)^* \gamma(x)w.
\end{align}
\end{defi}

As observed in \cite[\Sec 3]{NiWi1995} for irreducible semidiscrete and not necessarily depth 2 subfactors:

\begin{prop}
The subfactor theoretical Fourier transform $\cF$ is injective.
\end{prop}

For every $\phi\in\UCP_\N(\M)$, define $V_\phi$ as the closure of the operator:
\begin{align}\label{eq:defV}
V_\phi y \Omega= \phi(y) \Omega,\quad y\in\M
\end{align}
where $\Omega\in \Hil$ is a cyclic and separating unit vector for $\M$ such that the associated state $\omega = (\Omega , \slot \Omega)$ on $\M$ is $E$-invariant. In particular, $V_E$ is the Jones projection $e_\N$ for $\N\subset\M$ with respect to $E$. As in \cite[\Sec 2]{NiStZs2003}, \cite[\Sec 2.5]{BiDeGi2021}, it follows that $V_\phi\in\B(\Hil)$, $\|V_\phi\| = 1$ and $V_\phi \Omega = \Omega$. 

\begin{lem}
The operator $V_\phi$ depends only on $\phi$ and on the positive cone of $\Omega$. \footnote{From now on the vectors $\Omega$ and $\xi$ will be chosen in the same positive cone with respect to $\M$.}
\end{lem}

\begin{proof}
The proof of \cite[\Lem A]{Ko1989}, which shows that the Jones projection of $E$ depends only on the positive cone, adapts to an arbitrary $\N$-bimodular $\phi$. 
\end{proof}

Assume for the moment that $\N\subset\M$ is in addition discrete and local, see Remark \ref{rmk:Vinrelcomm}. Then $V_\phi \in \N' \cap \M_1$. Indeed, $\phi$ is $\N$-bimodular by assumption, thus $V_\phi \in \N'$, and $V_\phi = J_{\M,\Omega} V_\phi J_{\M,\Omega}$ holds because $\phi$ is automatically $\Omega$-adjointable \cite[\Lem 4.22]{BiDeGi2021} in the sense of Definition \ref{def:Omegaadj}, thus $V_\phi \in \M_1$. In fact, $J_{\M,\Omega} V_\phi =  V_\phi J_{\M,\Omega}$ is equivalent to $\Omega$-adjointability \cite[\Sec 6]{AcCe1982}, \cite[\Sec 2.5]{BiDeGi2021}.

Note that $\phi^\sharp$ is $\N$-bimodular when $\phi$ is $\N$-bimodular. Moreover,
\begin{align}\label{eq:propsV}
V_{\phi_1 \circ \phi_2} = V_{\phi_1} V_{\phi_2}, \quad V_{\phi^\sharp} = V_{\phi}^*.
\end{align}

\begin{rmk}\label{rmk:Vinrelcomm}
As observed in \cite[\Lem 4.22]{BiDeGi2021}, every $\N$-bimodular ucp map $\phi:\M\to\M$ is $\Omega$-adjointable under the weaker condition $a_\rho = 1_{H_\rho}$ for every irreducible $\rho \prec \theta$, where the operators $a_\rho$ are introduced in \cite[\Sec 3]{IzLoPo1998}. This condition is implied \eg by discreteness and locality \cite[\Prop 2.19]{BiDeGi2021} and by finiteness of the index \cite[\Sec 3]{IzLoPo1998}. In these cases $V_\phi \in \N' \cap \M_1$.
\end{rmk}

If $\phi\in \Span_\CC(\UCP_\N(\M))$, namely if $\phi = \sum_j \alpha_j \phi_j$ with $\alpha_j\in\CC$, $\phi_j\in\UCP_\N(\M)$, then $V_\phi$ defined as in \eqref{eq:defV} is bounded and it belongs to $\N'\cap\M_1$. It holds $V_\phi = \sum_j \alpha_j V_{\phi_j}$, the $\Omega$-adjoint operation \eqref{eq:Omegaadj} extends antilinearly, and the properties stated in \eqref{eq:propsV} continue to hold.

\begin{defi}\label{def:Fhat}
We define the Fourier transform on the complex span of $\N$-bimodular ucp maps $\M\to\M$ as follows:
\begin{align}
\widehat{\cF}:{{{\Span}_{\CC}}({\UCP}_{\N}}&(\M))\to \Hom(\theta,\theta)\\
&\phi \mapsto \gamma_1(V_\phi)
\end{align}
as $\gamma_1 (\N' \cap \M_1) = \theta(\N)' \cap \N = \Hom(\theta,\theta)$.
\end{defi}

\begin{prop}
The map $\widehat \cF$ is injective.
\end{prop}

\begin{proof}
The vector $\Omega$ is separating for $\M$, and $V_{\phi} = V_{\phi'}$ implies $\phi(y)\Omega = \phi'(y)\Omega$ for every $y\in\M$.
\end{proof}

For every $x\in \Hom(\gamma,\gamma)$ positive and such that $w^* x w = 1$, let $\phi_x := w^*x\gamma(\cdot)w \in \UCP_\N(\M)$. 
Note that $w^* x w$ is a multiple of $1$ whenever $x \in \Hom(\gamma,\gamma)$.
Then $\phi_x$ is dominated by $E$, namely $d E - \phi_x$ is completely positive with $d := \|x\| >0$.
Moreover, every $\phi \in \UCP_\N(\M)$ dominated by $E$ is of the form $\phi = \phi_x$ by an $L^\infty$ version of the Radon--Nikodym theorem for completely positive maps \cite[\Prop 1.4.2]{Ar1969}, \cite[\Prop 5.4]{Pa1973}, \cite[\Prop A.5]{Bi2016}.
More generally, every $x\in \Hom(\gamma,\gamma)$ can be written as $x=\sum_k \alpha_k x_k$ with $\alpha_k\in\CC$, $x_k$ positive and $w^* x_k w = 1$. Then $\phi_x := w^*x\gamma(\cdot)w \in \Span_\CC(\UCP_\N(\M))$.

\begin{rmk}\label{rmk:embeddingiflocal}
If $\N\subset \M$ is discrete and local, by Proposition \ref{prop:gammagammaisLinfty}, $\Hom(\gamma,\gamma) \cong L^\infty(K,\mu_E)$ and 
$$x\in \Hom(\gamma,\gamma) \mapsto \phi_x$$ 
corresponds to the embedding of functions $f \in L^\infty(K, \mu_E)$ into measures dominated by $\mu_E$, namely to $f \mapsto f \dd\mu_E$. Positive operators are mapped to positive measures, the condition $w^* x w = 1$ corresponds to $\int_K f \dd\mu_E = 1$.
\end{rmk}

The following proposition states that $\widehat\cF$ extends the subfactor theoretical Fourier transform:

\begin{prop}\label{prop:FhatisF}
$\widehat \cF$ extends $\cF$ from $\Hom(\gamma,\gamma)$ to $\Span_\CC(\UCP_\N(\M))$:
$$\widehat{\cF}(\phi_x)=\cF(x).$$
\end{prop}

\begin{proof}
Let $v_1\in\M_1$ be as in \cite[\Sec 2.5]{LoRe1995}, namely $v^\prime:n\xi\rightarrow n\Omega$, $n\in\N$ and $v_1:=\Ad_{J_{\M,\xi}}(v^\prime)$. We have $v_1v_1^*=e_\N$, $v_1\in\Hom(\id_{\M_1},\gamma_1)$ and $\gamma_1(v_1)=w$.
We have to show that
$$\gamma_1(V_{\phi_x})=\theta(w)^* \gamma(x)w$$
or equivalently
$$V_{\phi_x}=w^* x v_1.$$
Note that $v_1\Omega=w\Omega$ since $w^*v_1\Omega=\gamma_1(v_1^*)v_1\Omega=v_1 v_1^*\Omega=e_\N\Omega=\Omega.$
For every $y\in\M$, we have
$$w^* xv_1 y\Omega=w^* x\gamma(y)v_1\Omega=w^* x\gamma(y)w\Omega = \phi_x(y)\Omega$$
from which we get the claim.
\end{proof}

\begin{rmk}
The proposition above holds for arbitrary semidiscrete subfactors. The proof shows that $V_{\phi_x} \in \N' \cap \M_1$ for every $x \in \Hom(\gamma,\gamma)$, without the conditions mentioned in Remark \ref{rmk:Vinrelcomm}.
\end{rmk}

\subsection{The local discrete case: Fourier transform on measures}\label{sec:classicalF}

The subfactor theoretical Fourier transform $\cF$ and its extension $\widehat\cF$ can be defined for semidiscrete subfactors. Assume that $\N\subset\M$ is discrete and local. Then by Proposition \ref{prop:SpanUCP}, $\Span_\CC(\UCP_\N(\M))$ can be identified with the set of complex bounded Radon measures $M(K,\B(K))$ on the subfactor theoretical hypergroup $K$.
In this section, we check that $\widehat\cF$ defined on $\Span_\CC(\UCP_\N(\M))$ agrees with the classical Fourier transform $\mathfrak{F}$ defined on $M(K,\B(K))$.

The hypergroup theoretical convolution and involution on probability measures can be extended to $M(K,\B(K))$, endowing it with the structure of an involutive algebra, \cf \cite[\Rmk 3.3]{BiDeGi2021}. 

\begin{defi}
A representation of a compact hypergroup $K$ on a Hilbert space $H_\pi$ (\cf \cite[\Def 2.1.1]{BlHe1995}) is a unital involutive algebra homomorphism:
$$\pi: M(K,\B(K)) \to \B(H_\pi).$$ 

The representation is called continuous if its restriction to \emph{positive} measures is continuous from the weak* topology on $M(K,\B(K))$ to the weak operator topology on $\B(H_\pi)$. 
\end{defi}

\begin{defi}
Let $\mathfrak{F}$ be the hypergroup theoretical Fourier transform (\cf \cite[\Sec 3]{Vr1979}):
\begin{align}
\mathfrak{F}: M(K,\B(K)) &\to \bigoplus_{\pi} \B(H_\pi) \\
\mu &\mapsto \bigoplus_{\pi} \int_K \pi(k) \dd\mu(k)
\end{align}
where the direct sum is over all unitary equivalence classes $[\pi]$ of continuous irreducible representations of $K$ and $\pi(k) := \pi(\delta_k)$, $k\in K$. 
\end{defi}
Note that  $\mathfrak{F}$ depends on the choice of representative in each unitary equivalence class and that different choices yield unitarily equivalent Fourier transforms.

\begin{rmk}
If $\Gamma$ is a finite abelian group with an outer action on $\N$, the subfactor theoretical Fourier transform considered in Definition \ref{def:subfF} for the crossed product $\N \subset \N \rtimes \Gamma = \M$ corresponds in fact to the \emph{inverse} of the ordinary group theoretical Fourier transform. Namely to the map which associates to a function $f$ on the Pontryagin dual $\widehat \Gamma = G$, the function $\chi\in\Gamma \mapsto \int_G f(g) \chi(g) \dd g$ instead of $\int_G f(g) \overline\chi(g) \dd g$. 
Hence we may have used the symbol $\cF^{-1}$ in Definition \ref{def:subfF} instead of $\cF$.
\end{rmk}

By \cite[\Thm 6.4]{BiDeGi2021}, each $[\pi]$ admits a representative $\pi_\rho$ on the space of charged fields $H_\rho$ (whose dimension is finite and equal to the multiplicity $n_\rho$ of $\rho\prec\theta$) defined by
\begin{align}
\pi_\rho : M(K,\B(K)) &\to \B(H_\rho) \\
\pi_\rho(\mu) \psi &:= \phi_\mu(\psi)
\end{align}
where $\mu \mapsto \phi_\mu$ is the identification of $M(K,\B(K))$ with $\Span_\CC(\UCP_\N(\M))$, the inverse of the map $\phi \mapsto \mu_\phi$ of Proposition \ref{prop:SpanUCP}.

\begin{prop}\label{prop:intertHrhoHomrhobar}
Let $H_\rho$ be endowed with the inner product $(\psi_1, \psi_2) := E(\psi_2 \psi_1^*)$ and let $\Hom(\bar\rho,\theta)$ be endowed with the inner product 
$(w_1, w_2) := w_1^*w_2$. For every $\psi\in H_\rho$, let $\psi^\bullet := \psi^* \rbar_\rho \in H_\rhobar$ where $\rbar_\rho \in \Hom(\id_\N, \rho\rhobar)$ is part of a standard solution of the conjugate equations for $\rho$ and $\rhobar$.

Then the linear map:
$$\psi\in H_\rho \mapsto \gamma(\psi^{\bullet *})w \in \Hom(\bar\rho,\theta)$$
is a unitary intertwiner between the representation $\pi_\rho$ acting on $H_\rho$ as above and the representation $U_\rho$ acting on 
$\Hom(\rhobar,\theta)$ 
as follows 
$$U_\rho(\mu)\gamma(\psi^{\bullet *})w := \gamma(\phi_\mu(\psi^{\bullet *}))w.$$
\end{prop}

\begin{proof}
The map preserves the inner products on $H_\rho$ and $\Hom(\rhobar,\theta)$ respectively: 
\begin{align}
(\gamma(\psi_1^{\bullet *})w, \gamma(\psi_2^{\bullet *})w) &= w^*\gamma(\psi_1^{\bullet}\psi_2^{\bullet *})w \\
&= E(\psi_1^{\bullet}\psi_2^{\bullet *}) \\
&= E(\psi_2 \psi_1^*) = (\psi_1, \psi_2)
\end{align}
where for the third equality we refer to \cite[\Sec 3]{IzLoPo1998} and \cite[\Prop 2.19]{BiDeGi2021}.

It is surjective by \cite[\Lem 6.15]{DeGi2018} and it intertwines the representations $\pi_\rho$ and $U_\rho$, since
\begin{align}
\gamma((\pi_\rho(\mu)\psi)^{\bullet *}) w &= \gamma(\phi_\mu(\psi)^{\bullet *}) w \\
&= \gamma(\phi_\mu(\psi^{\bullet *})) w \\
&= U_\rho(\mu)\gamma(\psi^{\bullet *}) w
\end{align}
where the second equality holds by $\N$-bimodularity of $\phi_\mu$.
\end{proof}

Below, we choose the representatives of the unitary equivalence classes of continuous irreducible representations of $K$ to be the $U_\rho$ considered above and we take the corresponding $\mathfrak{F}$.

\begin{prop}
\begin{align}
\widehat{\mathcal{F}}(\phi_\mu) = \mathfrak{F}(\mu).
\end{align}
\end{prop}

\begin{proof}
On the one hand, for every $\phi\in \Span_\CC(\UCP_\N(\M))$, $\psi \in H_\rho$, $m\in \M$, and $v_1$, $\Omega$ as in the proof of Proposition \ref{prop:FhatisF}, we have 
\begin{align}
V_\phi \psi^* v_1 m\Omega &=V_\phi \psi^* \gamma(m) v_1\Omega \\
&= V_\phi \psi^* \gamma(m) w\Omega \\
&= \phi(\psi^*)\gamma(m) w\Omega \\
&= \phi(\psi^*) v_1 m\Omega
\end{align}
from which $V_\phi \psi^* v_1=\phi(\psi^*) v_1$, as $\Omega$ is cyclic for $\M$. Thus 
\begin{align}
\widehat{\mathcal{F}}(\phi) \gamma(\psi^{\bullet *}) w &= \gamma_1(V_{\phi}) \gamma(\psi^{\bullet *}) w \\
&= \gamma_1(V_{\phi} \psi^{\bullet *} v_1) \\
&= \gamma(\phi(\psi^{\bullet *})) w.
\end{align}

On the other hand, by Proposition \ref{prop:intertHrhoHomrhobar} and \cite[\Prop 5.5]{BiDeGi2021}, for every $\mu \in M(K,\B(K))$
\begin{align}
\mathfrak{F}(\mu) \gamma(\psi^{\bullet *}) w &= \int_K U_\rho(k) \gamma(\psi^{\bullet *}) w  \dd\mu(k) \\
&= \int_K \gamma( \phi_{\delta_k}(\psi^{\bullet *})) w \dd\mu(k) \\
&= \gamma( \phi_{\mu}(\psi^{\bullet *})) w
\end{align}
concluding the proof.
\end{proof}

\subsection{$L^p$ spaces and Fourier inequalities}\label{sec:Finequalities}

For the remaining part of the paper, we assume that $\N\subset\M$ is in addition \emph{discrete} and \emph{local}. By Proposition \ref{prop:gammagammaisLinfty}, $\Hom(\gamma,\gamma)$ is commutative and *-isomorphic to $L^\infty(K, \mu_E)$. 

\begin{lem}\label{lem:Estate}
The functional $x\mapsto w^*xw$ on $\Hom(\gamma,\gamma)$ coincides with the restriction of $E$ to $\Hom(\gamma,\gamma)$, in particular it is a normal faithful state. 
\end{lem}

\begin{proof}
This observation is due to \cite[\Cor 3]{NiWi1995}. Notice first that $w^*xw$ and $E(x)$ are both numbers (scalar multiples of 1) for every $x\in\Hom(\gamma,\gamma)$. Thus $w^*xw = w^* \gamma(x) w = E(x)$ follows. 
\end{proof}

\begin{notation}
Denote $L^\infty(K) := \Hom(\gamma,\gamma)$ and $M(K) := \Span_\CC(\UCP_\N(\M))$. Denote also $P(K) := \UCP_\N(\M)$.
\end{notation}

\begin{defi}
Let $L^1(K)$, $L^2(K)$ and more generally $L^p(K)$, $1\leq p < \infty$, be the completion of $\Hom(\gamma,\gamma)$ in the norm $\|x\|_p := (w^* |x|^p w)^{1/p} = E(|x|^p)^{1/p}$, with $|x| = (x^*x)^{1/2} \in \Hom(\gamma,\gamma)$ the modulus of $x$. Let $\|x\|_\infty := \|x\|$ be the operator norm on $\Hom(\gamma,\gamma)$, or equivalently $\|x\|_\infty := \|x\|_{\B(L^2(K))}$.
\end{defi}

Recall that if $\N\subset\M$ is irreducible and discrete, by \cite[\Thm 3.3]{IzLoPo1998}, $\Hom(\theta,\theta)$ is *-isomorphic to a von Neumann algebraic direct sum of matrix algebras:
$$\Hom(\theta,\theta) \cong \bigoplus_{[\rho]} M_{n_\rho}(\CC)$$
where $\rho$ runs over the inequivalent irreducible (hence with finite dimension $d(\rho)$) subendomorphisms of $\theta = \bigoplus \rho$ and $n_\rho$ is the multiplicity of $\rho$ in $\theta$. The index of the subfactor is finite if and only if the decomposition $\theta = \bigoplus \rho$ has finitely many summands.

Denote by $\Tr$ the canonical tracial weight on $\Hom(\theta,\theta)$ normalized such that $\Tr(1_{M_{n_\rho}(\CC)}) = n_\rho d(\rho)$. Then $\Tr$ is normal faithful and semifinite. Denote by $m_{\Tr}$ its domain.

\begin{defi}
Let $L^p(\widehat K)$, $1\leq p < \infty$, be the completion of $m_{\Tr}$ in the norm $\|x\|_p := \Tr(|x|^p)^{1/p}$, where $|x| = (x^*x)^{1/2} \in m_{\Tr}$.
Let $L^\infty(\widehat K) := \Hom(\theta,\theta)$ with the operator norm in the GNS representation with respect to $\Tr$, \ie $\|x\|_\infty := \|x\|_{\B(L^2(\widehat K))}$, 
or equivalently $\|x\|_\infty := \|x\|$ because $\Tr$ is normal faithful and semifinite. 
\end{defi}

For the classical (tracial) theory of noncommutative integration we refer to \cite{Ne1974}, \cite{TePhD}.

\begin{rmk}\label{rmk:TrisEhat}
By the proof of \cite[\Lem 7.3]{BiDeGi2021}, if $\N\subset\M$ is irreducible discrete and local (or finite index, or if it fulfills the condition $a_\rho = 1_{H_\rho}$ where the operators $a_\rho$ are introduced in \cite[\Sec 3]{IzLoPo1998}),
then $\Tr$ on $\Hom(\theta,\theta)$ coincides with $\gamma \circ \widehat E \circ \gamma_1^{-1}$ where $\widehat E : \M_1 \to \M \subset\M_1$ is the dual operator-valued weight of $E$, here restricted to $\N'\cap\M_1$.
\end{rmk}

By \cite[\Prop 4.15]{BiDeGi2021} the Fourier transform gives a one to one correspondence between \lqq trigonometric polynomials\rqq $\psi_{\rho,r}^* \bar\psi_{\rho,s}$ in $\Hom(\gamma,\gamma)$ and \lqq matrix units\rqq $w_{\rho,r} w_{\rho,s}^*$
in $\Hom(\theta,\theta)$. Namely,
$$\cF(\psi_{\rho,r}^* \bar\psi_{\rho,s}) = w_{\rho,r} w_{\rho,s}^*.$$
Trigonometric polynomials and matrix units are defined as follows:

\begin{notation}\label{not:piporhor}
Choose a Pimsner--Popa basis of charged fields $\{\psi_{\rho,r}\}$ for $\N\subset\M$ (see Section \ref{sec:condexp}) labelled by the inequivalent irreducible subendomorphisms $\rho \prec \theta$ and by a multiplicity counting index $r = 1,\ldots, n_\rho$.

Matrix units are then defined by $w_{\rho,r} w_{\rho,s}^*$ where $w_{\rho,r} := \gamma(\psi_{\rho,r}^*)w$ is an isometry in $\Hom(\rho,\theta)$. They have been exploited in \cite[\Sec 4.1]{BiDeGi2021}. 

Trigonometric polynomials are defined by $\psi_{\rho,r}^* \bar\psi_{\rho,s}$ where $\bar\psi_{\rho,s}$ is a so-called dual field. Namely, $\bar\psi_{\rho,s} := w_{\rho,s}^* m = w^* \bar\iota(\psi_{\rho,s})m$ where $m := \sum_{\rho,r} \theta(w_{\rho,r})\gamma(\psi_{\rho,r})$ is only a formal sum when $\theta = \bigoplus \rho$ is an infinite direct sum. In \cite[\Sec 2.4]{BiDeGi2021}, it is shown that 
$\bar\psi_{\rho,s}$ is a well defined operator in $\N$ and it belongs to $\Hom(\iotabar,\rho\iotabar)$. 
\end{notation}

The matrix units are dense in $\Hom(\theta,\theta) \cong \bigoplus_{[\rho]} M_{n_\rho}(\CC)$ in the weak operator topology, hence:

\begin{prop}\label{prop:Fhasdenserange}
The subfactor theoretical Fourier transform $\cF$ has dense range, \ie, in the previous notation, $\cF(L^\infty(K))$ is dense in $L^\infty(\widehat K)$ in the weak operator topology.
\end{prop}

\begin{prop}[Parseval's identity]\label{prop:Parseval}
The Hilbert spaces $L^2(K)$ with inner product defined by completion of $(x|y)_{L^2(K)} := w^* x^* y w = E(x^*y)$, $x,y\in \Hom(\gamma,\gamma)$, and the Hilbert space $L^2(\widehat K)$ with inner product defined by completion of $(x|y)_{L^2(\widehat K)} := \Tr(x^*y)$, $x,y\in m_{\Tr}$, are isomorphic via the Fourier transform:
$$(x|y)_{L^2(K)} = (\cF(x)|\cF(y))_{L^2(\widehat K)}.$$
In particular,
$$\|x\|_2 = \|\cF(x)\|_2.$$
\end{prop}

\begin{proof}
By density, it is enough to check the equality of the two inner products on trigonometric polynomials $\psi_{\rho,r}^* \bar\psi_{\rho,s}$ and matrix units $\cF(\psi_{\rho,r}^* \bar\psi_{\rho,s}) = w_{\rho,r} w_{\rho,s}^*$.

The inner product on trigonometric polynomials in $L^2(K)$ reads:
\begin{align}
w^* \bar\psi_{\rho,s}^* \psi_{\rho,r} \psi_{\rho',r'}^* \bar\psi_{\rho',s'} w &= w^* \bar\psi_{\rho,s}^* E(\psi_{\rho,r} \psi_{\rho',r'}^*) \bar\psi_{\rho',s'} w\\
&= \delta_{\rho,\rho'} \delta_{r,r'} w^* m^* w_{\rho,s} w_{\rho,s'}^*m  w\\
&= \delta_{\rho,\rho'} \delta_{r,r'} w^* \sum_{\sigma',t'} \gamma(\psi_{\sigma',t'}^*) \theta(w_{\sigma',t'}^*)  w_{\rho,s} w_{\rho,s'}^*\sum_{\sigma'',t''} \theta(w_{\sigma'',t''})\gamma(\psi_{\sigma'',t''}) w\\
&= \delta_{\rho,\rho'} \delta_{r,r'} \sum_{\sigma,t} E(\psi_{\sigma,t}^* \psi_{\rho,s}^*) E(\psi_{\rho,s'}\psi_{\sigma,t})\\
&= \delta_{\rho,\rho'} \delta_{r,r'} \sum_{\sigma,t} E(\psi_{\rho,s}^* \psi_{\sigma,t}^*) E(\psi_{\sigma,t} \psi_{\rho,s'})\\
&= \delta_{\rho,\rho'} \delta_{r,r'} \sum_{\sigma',\sigma'',t',t''} E(E(\psi_{\rho,s}^* \psi_{\sigma',t'}^*) \psi_{\sigma',t'}\psi_{\sigma'',t''}^* E(\psi_{\sigma'',t''} \psi_{\rho,s'}))\\
&= \delta_{\rho,\rho'} \delta_{r,r'} E(\psi_{\rho,s}^* \psi_{\rho,s'})\\
&= \delta_{\rho,\rho'} \delta_{r,r'} \delta_{s,s'} d(\rho)
\end{align}
where we used the definition of the dual fields $\bar\psi_{\rho,s} = w_{\rho,s}^* m$ with $m = \sum_{\sigma,t} \theta(w_{\sigma,t})\gamma(\psi_{\sigma,t})$, the intertwining and orthonormality properties $w_{\sigma,r} \in \Hom(\sigma,\theta)$, $w_{\sigma,r}^*w_{\sigma',r'} = \delta_{\sigma,\sigma'} \delta_{r,r'} 1$, the locality commutation relations $\psi_{\rho,r}\psi_{\sigma,t} = \varepsilon_{\sigma,\rho} \psi_{\sigma,t} \psi_{\rho,r}$, the Pimsner--Popa expansion \cite[\Sec 2.2]{BiDeGi2021} and the normalization $\psi_{\rho,s}^* \psi_{\rho,s'} = \delta_{s,s'} d(\rho) 1$ \cite[\Sec 2.3]{BiDeGi2021}.

The inner product on matrix units in $L^2(\widehat K)$ also reads:
\begin{align}
\Tr((w_{\rho,r} w_{\rho,s}^*)^* (w_{\rho',r'} w_{\rho',s'}^*)) &= \Tr(w_{\rho,s} w_{\rho,r}^* w_{\rho',r'} w_{\rho',s'}^*)\\
&= \delta_{\rho,\rho'} \delta_{r,r'} \Tr(w_{\rho,s} w_{\rho,s'}^*)\\
&= \delta_{\rho,\rho'} \delta_{r,r'} \delta_{s,s'} d(\rho)
\end{align}
by the choice of normalization of $\Tr$ on $M_{n_\rho}(\CC)$.
\end{proof}

\begin{rmk}
For irreducible discrete depth 2 subfactors, the statements of the two previous propositions appear in \cite[\Thm 17]{NiWi1995}. 
\end{rmk}

\begin{prop}\label{prop:RiemannLebesgue}
$$\|\cF(x)\|_\infty \leq \|x\|_1.$$
\end{prop}

\begin{proof}
For positive elements $x$ in $\Hom(\gamma,\gamma)$ normalized such that $w^* x w = 1$, we have 
$$\|\cF(x)\|_\infty = \|\widehat\cF(\phi_x)\|_\infty = \|\gamma_1(V_x)\| = \|V_x\| = 1$$ 
by Proposition \ref{prop:FhatisF}, and $\|x\|_1 = 1$ by definition. Thus on positive elements $\|\cF(x)\|_\infty = \|x\|_1$, \footnote{In the case of the classical Fourier transform this holds as $\|f\|_1 = \int f(x) \dd x = \widehat f(0) \leq \|\widehat f\|_\infty \leq \|f\|_1$ for positive $f$.}. For an arbitrary element $x$ in $\Hom(\gamma,\gamma)$ we need its identification with a function $f$ in $L^\infty(K, \mu_E)$ and a modification of the Hahn-Jordan decomposition theorem adapted to complex bounded measures. Let $f = \nu |f|$ be the polar decomposition of $f$ with $\nu\in L^\infty(K, \mu_E)$, $|\nu| =1$. Let $\nu_n\in L^\infty(K, \mu_E)$, be a uniform (by boundedness) approximation of $\nu$ by simple measurable functions \cite[\Thm 1.17]{RudRC}. As $|\nu_n| \to |\nu| = 1$, $n\to \infty$, we can divide and assume $|\nu_n| = 1$ for every $n$. Define $f_n := \nu_n |f|$ or equivalently $f_n := \sum_{m} \nu_n^m \chi_{K_n^m} |f|$ where $\nu_n^m$ runs over the finitely many different values of each $\nu_n$ and $\chi_{K_n^m}$ is the characteristic function of $K_n^m := \{k\in K:\, \nu_n(k) = \nu_n^m\}$. Fixed $n$, the sets $K_n^m$ are pairwise disjoint and $|f| = |f_n|= \sum_m \chi_{K_n^m} |f|$. Then
$$\|\cF(f_n)\|_\infty = \|\sum_{m} \nu_n^m \cF(\chi_{K_n^m} |f|)\|_\infty \leq \sum_{m} \|\chi_{K_n^m} |f|\|_1 = \|f\|_1$$
and by $\|\cF(f_n - f) \|_\infty \leq \|f_n - f\|_\infty$, \footnote{The inequality $\|\widehat f\|_\infty \leq \|f\|_\infty$ holds in the case of the classical Fourier transform because we are integrating with respect to a probability measure.}, where the $L^\infty$-norms both coincide with the operator norm in $\B(\Hil)$, together with $f_n \to f$ in $L^\infty(K, \mu_E)$ we get the statement.
\end{proof}

\begin{rmk}
The previous proposition is also a consequence of the identification of the subfactor theoretical Fourier transform with the classical Fourier transform on subfactor theoretical compact hypergroups, see Section \ref{sec:classicalF}. Note that in the previous proof we only need the fact that $\Hom(\gamma,\gamma)$ is commutative, thus identified with functions on a probability space, not its hypergroup structure.
\end{rmk}

\begin{rmk}
The weaker bound $\|\cF(x)\|_\infty \leq \|x\|_2$ can be proven without the identification of $\Hom(\gamma,\gamma)$ with $L^\infty(K, \mu_E)$ as follows.
By the \Cstar-identity, for every $x\in\Hom(\gamma,\gamma)$ it holds
\begin{align}
\|\cF(x)\|_\infty &= \|\cF(x)\cF(x)^*\|^{1/2}\\
&= \|\theta(w)^* \gamma(x) ww^* \gamma(x)^* \theta(w)\|^{1/2}\\
&\leq \|\theta(w)^* \gamma(xx^*) \theta(w)\|^{1/2}\\
&= \|\gamma(w^* xx^* w)\|^{1/2}\\
&= (w^* x^*x w)^{1/2} = \|x\|_2 
\end{align}
because $ww^*$ is a projection, thus $ww^* \leq 1$, and $0 \leq a \leq b$ implies $\|a\| \leq \|b\|$, because $\gamma$ is isometric (unital is enough) and $xx^* = x^*x$ by commutativity. The same proof holds for finite index irreducible subfactors, not necessarily local. Indeed, by \cite[\Lem 3.7]{LoRo1997}, \cite[\Prop 2.4]{BiKaLoRe2014-2}, \cite[\Prop 8.33]{GiLo2019}, $x \mapsto w^*xw$ is a trace on $\Hom(\gamma,\gamma)$, which needs no longer be commutative. See also \cite[\Prop 4.7]{JiLiWu2016} for a proof of $\|\cF(x)\|_\infty \leq \|x\|_1$ for arbitrary finite index irreducible subfactors.
\end{rmk}

Thanks to the alternative description of $L^p(K)$ and $L^p(\widehat K)$ as complex interpolation spaces \cite{Ko1984}, \cite{Te1982}, in the special case of tracial states and weights, by \cite[\Thm 1.2, \Rmk 3.4]{Ko1984} (see references therein) and by Proposition \ref{prop:Parseval} and Proposition \ref{prop:RiemannLebesgue} we get:

\begin{prop}[Hausdorff--Young inequality]\label{prop:Hausdorff-Young}
$$\|\cF(x)\|_p \leq \|x\|_q$$
for $2\leq p\leq \infty$, $1\leq q\leq 2$ and $1/p + 1/q = 1$.
\end{prop}

\begin{rmk}
The Hausdorff--Young inequality is a classical result for the Fourier analysis on groups. Recent proofs of the inequality for DJS hypergroups appear in \cite{De2013} for commutative hypergroups and in \cite{KuSa2020}, \cite{KuRu2020-arxiv}.
\end{rmk}

\subsection{Involutions, convolutions and products}\label{sec:invconv}

On the von Neumann algebra $L^\infty(K) (= \Hom(\gamma,\gamma))$ we have the ordinary unital *-algebra structure given by $(1,\cdot,{}^*)$, namely the unit operator, the multiplication and adjoint operations in $\B(\Hil)$. Likewise on $L^\infty(\widehat K) (= \Hom(\theta,\theta))$. In the \emph{absence} of a globally defined inverse subfactor theoretical Fourier transform $\cF^{-1} : \Hom(\theta,\theta) \to \Hom(\gamma,\gamma)$ for infinite index subfactors, see Remark \ref{rmk:inverseF}, we use the embedding $x \in L^\infty(K) \mapsto \phi_x := w^* x \gamma(\slot) w \in M(K) (= \Span_\CC(\UCP_\N(\M)))$ to give $L^\infty(K)$ a second *-algebra structure. In Proposition \ref{prop:Fisstarhomo}, we show that this second *-algebra structure has the right properties with respect to $\cF$. 

\begin{defi}
For $\phi_1$, $\phi_2 \in M(K)$, let $\phi_1 \ast \phi_2\in M(K)$ and $\phi_1^\sharp\in M(K)$ be the convolution and the involution of measures, defined respectively by the composition of ucp maps $\phi_1 \circ \phi_2$ and by the $\Omega$-adjoint of $\phi_1$ considered in Definition \ref{def:Omegaadj}.
\end{defi}

\begin{defi}\label{def:convoinvoinLinftyK}
For $x,y\in L^\infty(K)$, let $x \ast y := w^* x \gamma(y) w = w^* \gamma(y) x w \in L^\infty(K)$, namely the operator corresponding to $\phi_x \ast \phi_y$, and let $x^\sharp \in L^\infty(K)$ be the operator corresponding to $\phi_x^\sharp$.

We call $x\ast y$ \textbf{convolution} and $x^\sharp$ \textbf{involution} in $L^\infty(K)$.
\end{defi}

Inside $L^\infty(\widehat K) \cong \bigoplus_{[\rho]} M_{n_\rho}(\CC)$ one can consider the *-subalgebra of finite rank operators denoted by $\Trig(\N\subset\M)$ in \cite[\Sec 4.1]{BiDeGi2021}. Note that $\Trig(\N\subset\M) \subset m_{\Tr}$. 

\begin{rmk}
Recall the Pimsner--Popa basis of charged fields $\{\psi_{\rho,r}\}$ from Notation \ref{not:piporhor}.
Let $m_{\rho,r} \in \Hom(\theta, \theta^2)$ be defined by $m_{\rho,r} := \theta(w_{\rho,r}) \gamma(\psi_{\rho,r})$.
The formal sum $m = \sum_{\rho,r} m_{\rho,r}$ (infinite when the index of the subfactor is infinite) together with its formal adjoint $m^*$ play the role of \emph{comultiplication} and \emph{multiplication} for the algebra object $\theta$ describing the extension $\N\subset\M$. 
\end{rmk}

In \cite[\Thm 4.13]{BiDeGi2021}, it is shown that $\Trig(\N\subset\M)$ is an associative unital (commutative by locality) *-algebra with the following operations:

\begin{defi}
For $x,y\in \Trig(\N\subset\M) \subset L^\infty(\widehat K)$, let $x \ast y := m^* a \theta(b) m \in \Trig(\N\subset\M)$ and $x^\bullet := \theta(w^*m^*) \theta(a^*) mw = w^*m^* \theta(a^*) \theta(mw) \in \Trig(\N\subset\M)$ be respectively the multiplication and the involution in $\Trig(\N\subset\M)$. The unit is given by the Jones projection $e := ww^*$. 

We call $x\ast y$ \textbf{convolution} and $x^\bullet$ \textbf{involution} when $x$ and $y$ are thought of as elements in $L^\infty(\widehat K)$. 
\end{defi}

\begin{defi}
Denote by $\Trig(K)$ the set of $x\in L^\infty(K)$ such that $\cF(x)\in \Trig(\N\subset\M)$. 
\end{defi}

$\Trig(K)$ is dense in $L^\infty(K)$ in the weak operator topology and in $L^2(K)$ in the $L^2$-norm topology. 
Moreover, $\Trig(K) = L^\infty(K)$ if and only if the subfactor has finite index.

\begin{prop}\label{prop:Fisstarhomo}
For $x,y \in \Trig(K)$, we have
$$\cF(xy) = \cF(y) \ast \cF(x), \quad \cF(x^*) = \cF(x)^\bullet, \quad \cF(1) = e$$
where note that $xy = yx$ and $\cF(y) \ast \cF(x) = \cF(x) \ast \cF(y)$.

For $x,y \in L^\infty(K)$, we have
$$\cF(x\ast y) = \cF(x) \cF(y), \quad \cF(x^\sharp) = \cF(x)^*.$$
\end{prop}

\begin{proof}
For the first two equalities we refer to \cite[\Sec 4.3]{BiDeGi2021}. The remaining equalities follow by observing that
$$\cF(x\ast y) = \widehat \cF(\phi_x \ast \phi_y) = \gamma_1(V_{\phi_x} V_{\phi_y}) 
= \gamma_1(V_{\phi_x}) \gamma_1(V_{\phi_y}) = \cF(x) \cF(y)$$
and
$$\cF(x^\sharp) = \widehat \cF(\phi_x^\sharp) = \gamma_1(V_{\phi_x}^*) = \gamma_1(V_{\phi_x})^* = \cF(x)^*$$ 
by Proposition \ref{prop:FhatisF}. Moreover, $\cF(1) = \widehat \cF(E) = \gamma_1(v_1v_1^*) = ww^* = e$, where $v_1$ is the isometry in $\M_1$ splitting the Jones projection $v_1 v_1^* = e_\N$ as in the proof of Proposition \ref{prop:FhatisF}. 
\end{proof}

Note that the equalities $\cF(x\ast y) = \cF(x) \cF(y)$ and $\cF(1) = e$ can also be checked directly, without passing to $M(K)$, whereas the involution $x^\sharp$ cannot even be defined without it, to our knowledge.

\begin{rmk}
Note that $1\in \Trig(K)$. Instead, the convolution unit for $L^\infty(K)$ is not always an operator in $\Hom(\gamma,\gamma)$. It is the Dirac measure $\id \in P(K) (= \UCP_\N(\M))$ whose Fourier transform $\widehat \cF(\id) = 1$ sits in $L^\infty(\widehat K)$ but not in $\Trig(\N\subset\M)$, unless the subfactor has finite index.
\end{rmk}

\begin{rmk}\label{rmk:inverseF}
If $x\in \Trig(K)$, the equality $\cF(x^\sharp) = \cF(x)^*$ can be promoted to a definition of involution by means of the inverse subfactor theoretical Fourier transform, $x^\sharp := \cF^{-1}(\cF(x)^*)$, \footnote{For finite index subfactors, it coincides with the $180^{\circ}$-rotation of the diagram for $x^*$.}. Similarly for the convolution $x\ast y := \cF^{-1}(\cF(x) \cF(y))$. However, for infinite index subfactors, $\cF^{-1}$ is only partially defined on $\Hom(\theta,\theta)$ by the formal expression $y \mapsto \cF^{-1}(y) := M^* y \gamma(M)$, where $M := \gamma^{-1}(m)$ and $m$ is as above. 
\end{rmk}

\subsection{Convolution inequalities}\label{sec:convoineq}

In this section, we investigate analytic properties of the convolution and involution operators in $L^\infty(K)$: positivity and norm inequalities.

\begin{lem}\label{lem:convispos}
If $x,y\in L^\infty(K)$ are positive, then $x \ast y$ is positive.
\end{lem}

\begin{proof}
It follows immediately from the definition $x \ast y = w^* x \gamma(y) w$. Alternatively, one can observe that if $\phi_z = w^* z \gamma(\slot) w$, $z\in L^\infty(K)$, is (completely) positive on $\M$, then $z$ is positive. Indeed, if $w^* z \gamma(t^*t) w$ is positive for every $t\in\M$, then $(\gamma(t) w \xi, z \gamma(t) w \xi) \geq 0$ for every $\xi \in \Hil$, and vectors of the form $\gamma(t) w \xi$ are total in $\Hil$ by minimality of the Connes--Stinespring representation of $E$. Thus $\phi_{x \ast y} = \phi_{x} \ast \phi_{y}$ entails positivity of $x \ast y$.
\end{proof}

\begin{lem}\label{lem:involistracepreserving}
It holds $w^* x^\sharp w = w^* x^* w = \overline{w^* x w}$ for every $x\in L^\infty(K)$.
\end{lem}

\begin{proof}
Observe that $w^* x^\sharp w = w^* x^\sharp \gamma(1) w = \phi_{x^\sharp}(1) = \phi_{x}^\sharp(1) = \overline{\phi_{x}(1)}$ where the last equality follows from the definition of $\Omega$-adjoint $(\Omega, \phi_{x}^\sharp(1) \Omega) = (\phi_{x}(1) \Omega, \Omega)$.
\end{proof}

\begin{lem}\label{lem:involisstarisom}
The involution $x \mapsto x^\sharp$ in $L^\infty(K)$ is an antilinear *-isomorphism.
\end{lem}

\begin{proof}
If $x,y\in\Trig(K)$, the equalities $(xy)^\sharp = x^\sharp y^\sharp$ and $(x^*)^\sharp = (x^\sharp)^*$ follow from Proposition \ref{prop:Fisstarhomo}, by observing that $(\cF(y)\ast\cF(x))^* = \cF(y)^*\ast\cF(x)^*$ and $(\cF(x)^\bullet)^* = (\cF(x)^*)^\bullet$, and by injectivity of the Fourier transform. $1^\sharp = 1$ holds because $\id^\sharp = \id$. If $x\in L^\infty(K)$, 
let $x_n\in \Trig(K)$ such that $x_n \to x$ in the weak operator topology and $\|x_n\| \leq \|x\|$ by Kaplansky's density theorem. Then $\phi_{x_n} \to \phi_{x}$ in the pointwise weak operator topology, and by the same argument as in \cite[\Rmk 4.27]{BiDeGi2021}, using $\|V_{\phi_{x_n}}\| \leq \|x_n\|_1 \leq \|x_n\|$, it follows that $\phi_{x_n}^\sharp \to \phi_{x}^\sharp$ in the pointwise weak operator topology. Thus $(\gamma(t) w \xi, x_n^\sharp \gamma(s) w \eta) \to (\gamma(t) w \xi, x^\sharp \gamma(s) w \eta)$ for every $t,s\in\M$, $\xi,\eta\in\Hil$. As observed above, vectors of the form $\gamma(t) w \xi$ are total in $\Hil$. Moreover, $\|x_n^\sharp\| = \|x_n\|$ for $x_n\in\Trig(K)$ because 
\begin{align}
\|x_n^\sharp\| &= \|x_n^\sharp\|_{\B(L^2(K))} \\
&= \sup_{\xi \in L^2(K)} \|x_n^\sharp \xi\|_2 \|\xi\|_2^{-1} \\
&= \sup_{\xi \in \Trig(K)} \|x_n^\sharp \xi^\sharp\|_2 \|\xi^\sharp\|_2^{-1} \\
&= \sup_{\xi \in \Trig(K)} \| (x_n \xi)^\sharp\|_2 \|\xi^\sharp\|_2^{-1} \\
&= \sup_{\xi \in \Trig(K)} \| x_n \xi\|_2 \|\xi\|_2^{-1} = \|x_n\|
\end{align}
by using that $\Trig(K) = \Trig(K)^\sharp$ is dense in $L^2(K)$ and Lemma \ref{lem:involistracepreserving}. We conclude that $x_n^\sharp \to x^\sharp$ in the weak operator topology and $(xy)^\sharp = x^\sharp y^\sharp$ and $(x^*)^\sharp = (x^\sharp)^*$ hold for every $x,y\in L^\infty(K)$. 
\end{proof}

By Lemma \ref{lem:involisstarisom}, we have that 
\begin{align}\label{eq:sharpabsval}
|x^\sharp |^2 = (x^{\sharp})^* x^\sharp = (x^*x)^\sharp = (|x|^2)^\sharp = |x|^\sharp |x|^\sharp
\end{align}
and thus $|x^\sharp| = |x|^\sharp$, as $|x|^\sharp$ is positive by the proof of Lemma \ref{lem:convispos}. 

\begin{lem}\label{lem:sharpisLpisomet}
It holds $\|x^\sharp\|_p = \|x\|_p$ for every $x\in L^\infty(K)$, $1\leq p \leq \infty$.
\end{lem}

\begin{proof}
The statement for $p=\infty$ follows by Lemma \ref{lem:involisstarisom} and by the spectral properties of the \Cstar-norm.
For $p=2$, we compute $\|x^\sharp\|_2 = (w^* (x^\sharp)^* x^\sharp w)^{1/2} = (w^* (x^* x)^\sharp w)^{1/2} = \|x\|_2$ by Lemma \ref{lem:involistracepreserving}. 
For $p=1$, one can use \eqref{eq:sharpabsval}. For $1\leq p<\infty$, it follows by observing that $(|x^\sharp|)^p = (|x|^\sharp)^p = (|x|^p)^\sharp$.
Indeed, $x^*x$ is positive, thus the involution commutes with the real continuous functional calculus of $x^*x$, in this case with the function $y \mapsto y^{p/2}$, again by Lemma \ref{lem:involisstarisom}.
\end{proof}

\begin{lem}\label{lem:totalvariationonpos}
For positive elements $x\in L^\infty(K)$, 
it holds $\|x\|_1 = \phi_x(1) = \|\phi_x\|$, where $\|\phi_x\|$ is the norm of $\phi_x$ as a bounded linear operator on $\M$.

More generally, if $x\in L^\infty(K)$, it holds $(\|x\|_p)^p = \phi_{|x|^p}(1) = \|\phi_{|x|^p}\|$ for every $1 \leq p < \infty$ and $\|x\|_{\infty} = \inf\{\lambda>0: E-\lambda^{-1}\phi_{|x|} \, \text{is completely positive}\}$.
\end{lem}

\begin{proof}
The statements for $1 \leq p < \infty$ follow immediately from the definitions and from the positivity of $\phi_x$ and $\phi_{|x|^p}$. 
For the last statement, it is enough to observe that $E- \lambda^{-1}\phi_{|x|} = w^*(1-\lambda^{-1}|x|)\gamma(\cdot)w$
is (completely) positive if and only if $1-\lambda^{-1}|x|$ is positive, where the only if part follows by the proof of Lemma \ref{lem:convispos}.
This is achieved for every $\lambda > \||x|\|_\infty = \|x\|_\infty$. 
\end{proof}

\begin{prop}[Young inequality]\label{prop:Young}
If $x,y \in L^\infty(K)$, then
$$\|x\ast y\|_r \leq \|x\|_p \|x\|_q$$
for $1\leq p,q,r \leq \infty$ such that $1/p + 1/q = 1/r + 1$.
\end{prop}

\begin{proof}
Let $r=1$, $p=1$, $q=1$. Assume first that $x$ and $y$ are positive in $L^\infty(K)$, then by Lemma \ref{lem:convispos} and Lemma \ref{lem:totalvariationonpos} we have
\begin{align}
\|x\ast y\|_1 &= \| \phi_{x\ast y}\| \\
&= \| \phi_{x} \ast \phi_{y}\| \\
&\leq \| \phi_{x}\| \|\phi_{y}\| = \|x\|_1 \|y\|_1. 
\end{align}
For general $x, y \in L^\infty(K)$, as in the proof of Proposition \ref{prop:RiemannLebesgue}, let $x_n \to x$ and $y_k \to y$ in the $L^\infty$-norm topology such that $x_n = \sum_m \nu_n^m x_n^m$, $|x| = |x_n| = \sum_m x_n^m$ and $y_k = \sum_h \mu_k^h y_k^h$, $|y| = |y_k| = \sum_h y_k^h$, the sums over $m$ and $h$ are finite, $\nu_n^m$, $\mu_k^h$ are complex phases and $x_n^m$, $y_k^h$ are positive in $L^\infty(K)$. Then
$$\|x_n\ast y_k\|_1 = \| \sum_{m,h} \nu_n^m \mu_k^h (x_n^m \ast y_k^h) \|_1 \leq \sum_{m} \| x_n^m\|_1 \sum_{h}\|y_k^h \|_1 = \|x\|_1\|y\|_1$$
by the Minkowski inequality and the previous step. Moreover, $x_n \ast y_n = w^* x_n \gamma(y_n) w \to x \ast y$ in the $L^\infty$-norm topology, hence in the $L^1$-norm topology, thus $\|x \ast y\|_1 \leq \|x\|_1 \|y\|_1$.

Let $r=\infty$, $p=1$, $q=\infty$. By \eg \cite[\eq (25)]{Ne1974}, Proposition \ref{prop:Parseval}, Proposition \ref{prop:Fisstarhomo} and Lemma \ref{lem:sharpisLpisomet}, we get 
\begin{align}
\|x\ast y\|_\infty &= \sup_{z\in L^\infty(K), \|z\|_1 \leq 1} | w^* (z^* (x \ast y)) w | \\
&= \sup_{z\in L^\infty(K), \|z\|_1 \leq 1} | \Tr (\cF(z)^* \cF(x \ast y))| \\
&= \sup_{z\in L^\infty(K), \|z\|_1 \leq 1} | \Tr ((\cF(x)^*\cF(z))^* \cF(y))| \\
&= \sup_{z\in L^\infty(K), \|z\|_1 \leq 1} | \Tr (\cF(x^\sharp \ast z)^* \cF(y))| \\
&= \sup_{z\in L^\infty(K), \|z\|_1 \leq 1} | w^* ((x^\sharp \ast z)^* y) w | \\
&\leq \sup_{z\in L^\infty(K), \|z\|_1 \leq 1} \|x^\sharp \ast z\|_1 \|y\|_\infty \leq \|x\|_1 \|y\|_\infty
\end{align}
where in the last line we used the H\"older inequality and $\|x^\sharp \ast z\|_1 \leq \|x^\sharp\|_1 \|z\|_1$. By a symmetric argument, it follows also $\|x\ast y\|_\infty \leq \|x\|_\infty \|y\|_1$.

By complex interpolation \cite[\Thm 1.2, \Def 3.1]{Ko1984} among the two previous cases, we get $\|x \ast y\|_p \leq \|x\|_1 \|y\|_p$ and $\|x \ast y\|_p \leq \|x\|_p \|y\|_1$ for $1\leq p \leq \infty$.

Let $r=\infty$ and $p, q$ such that $1/p + 1/q =1$. As before, we get
\begin{align}
\|x\ast y\|_\infty &= \sup_{z\in L^\infty(K), \|z\|_1 \leq 1} | w^* ((x^\sharp \ast z)^* y) w | \\
&\leq \sup_{z\in L^\infty(K), \|z\|_1 \leq 1} \|x^\sharp \ast z\|_p \|y\|_q \leq \|x\|_p \|y\|_q
\end{align}
where we used again the H\"older inequality and the previously derived $\|x^\sharp \ast z\|_p \leq \|x^\sharp\|_p \|z\|_1$. Thus we have shown $\|x\ast y\|_\infty \leq \|x\|_p \|y\|_q$ for $p, q$ such that $1/p + 1/q =1$.

Again by complex interpolation \cite[\Thm 1.2, \Def 3.1]{Ko1984} among the cases $r=\infty$, $1/p + 1/q =1$, and $r=p$, $q=1$, or $r=q$, $p=1$, we get the general statement.
\end{proof}

\begin{rmk}
For $x$ and $y$ positive in $L^\infty(K)$, it also holds 
\begin{align}
\|x\ast y\|_1 &= (\phi_x \ast \phi_y) (1) \\
&= \phi_x(1) \phi_y(1) = \|x\|_1 \|y\|_1.
\end{align}
\end{rmk}

\begin{cor}
The space $L^1(K)$ is a complex Banach algebra with involution.
\end{cor}

\begin{rmk}
This fact is known for locally compact KPC hypergroups \cite[\Sec 5]{KaPoCh2010}, thus for subfactor theoretical compact hypergroups \cite[\Sec 3]{BiDeGi2021}, which are contained in this class.  
\end{rmk}

\subsection{Inversion formula and uncertainty principles}\label{sec:inversionuncertainty}

In this section, we prove the inversion formula for the subfactor theoretical Fourier transform and an uncertainty principle relating the size of the support of $x\in \Hom(\gamma,\gamma)$ and of $\cF(x)\in \Hom(\theta,\theta)$. Recall the notation $L^\infty(K) = \Hom(\gamma,\gamma)$ and $L^\infty(\widehat K) = \Hom(\theta,\theta)$. Recall also that $\Hom(\theta,\theta) \cong \bigoplus_{[\rho]} M_{n_\rho}(\CC)$, where the sum runs over inequivalent irreducible $\rho \prec \theta$. For $x\in L^\infty(K)$, let
\begin{align}
(\cF(x))(\rho) &:= \sum_{r,s = 1, \ldots, n_\rho} w_{\rho,r} w_{\rho,r}^* \cF(x) w_{\rho,s} w_{\rho,s}^* \\ 
&= \sum_{r,s = 1, \ldots, n_\rho} (w_{\rho,r}^* \cF(x) w_{\rho,s}) w_{\rho,r} w_{\rho,s}^*
\end{align}
and $((\cF(x))(\rho))_{r,s} := w_{\rho,r}^* \cF(x) w_{\rho,s} \in \Hom(\rho,\rho) = \CC 1$. Thus $(\cF(x))(\rho) \in M_{n_\rho}(\CC)$ for every $\rho \prec \theta$. The support of $\cF(x)$, not to be confused with the support projection of $\cF(x)$ and denoted below by $\supp \cF(x)$, can be considered to be the set of inequivalent irreducible $\rho \prec \theta$ such that $(\cF(x))(\rho) \neq 0$. Let also $\chi_{\rho,r,s} := \psi_{\rho,r}^* \bar\psi_{\rho,s} \in L^\infty(K)$ be the trigonometric polynomials considered in the proof of Proposition \ref{prop:Parseval}. By \cite[\Prop 4.15]{BiDeGi2021} and Proposition \ref{prop:Parseval},
\begin{align}
(\chi_{\rho,r,s}| x)_{L^2(K)} &= (w_{\rho,r} w_{\rho,s}^*| \cF(x))_{L^2(\widehat K)} \\
&= \Tr(w_{\rho,s} w_{\rho,r}^* \cF(x)) \\
&= w_{\rho,r}^* \cF(x) w_{\rho,s} \Tr(w_{\rho,s} w_{\rho,s}^*) \\
&= ((\cF(x))(\rho))_{r,s} d(\rho)
\end{align}
Moreover, $\{d(\rho)^{-1/2} \chi_{\rho,r,s}\}_{\rho,r,s}$ is an orthonormal basis of $L^2(K)$. Thus we get the following:

\begin{prop}[Inversion formula]\label{prop:inversion}
If $x \in L^2(K)$, then
$$x = \sum_{\rho,r,s} ((\cF(x))(\rho))_{r,s} \chi_{\rho,r,s}$$
where the sum converges in the $L^2$-norm topology.
\end{prop}

In the case of compact groups \cite[\Thm 2.4]{ChNg2005} and compact DJS hypergroups \cite[\Thm 4.1]{AlAm2017}, from the inversion formula for the Fourier transform and from Parseval's identity one can derive the Donoho--Stark uncertainty principle:
\begin{align}\label{eq:DonohoStarkuncertainty}
1 \leq \mu_K(\supp(f)) \sum_{\rho \in \supp \widehat f} n_\rho k_\rho
\end{align}
where $f\neq 0$ is a function in $L^2(K,\mu_K)$ and $\widehat f$ is its Fourier transform, $\mu_K$ is the Haar measure on the compact group or DJS hypergroup, $n_\rho$ is the dimension of the irreducible representation $\rho$ and $k_\rho$ is its hyperdimension \cite{Vr1979}, \cite{AmMe2014}.
Note that for compact groups $n_\rho = k_\rho$, and for subfactor theoretical compact hypergroups $n_\rho \leq k_\rho = d(\rho)$ \cite[\Cor 2.21, \Thm 6.5]{BiDeGi2021}. 

We prove a stronger version of the uncertainty principle \eqref{eq:DonohoStarkuncertainty} for local discrete subfactors, similar to the stronger version proved in \cite[\Thm 2]{AlRu2008} for compact groups. The dimension of the representation $n_\rho$ gets replaced with the rank of the matrix $(\cF(x))(\rho) \in M_{n_\rho}(\CC)$.
Using Proposition \ref{prop:Parseval} and Proposition \ref{prop:RiemannLebesgue}, we reformulate and prove the stronger uncertainty principle following the same argument used in \cite[\Thm 5.2]{JiLiWu2016}, \cite[\Prop 3.3]{LiWu2017} and \cite[\Thm 4.8]{LiPaWu2019-arxiv} respectively for finite index subfactors, fusion bialgebras and Kac-type compact quantum groups.
For an element $x$ in a von Neumann algebra $\A \subset \B(\Hil)$, denote by $[x]$ its support projection, \ie the smallest projection in $\B(\Hil)$ such that $x [x] = x$. Then $[x] \in \A$ and $[x] = [|x|]$. 

\begin{prop}
Denote by $\tau$ either the state $x \mapsto w^*xw$ on $L^\infty(K)$, \ie the restriction of $E$ to $\Hom(\gamma,\gamma)$, or the tracial weight $\Tr$ on $L^\infty(\widehat K)$. For every $x \in L^\infty(K)$,
$x\neq 0$, it holds
$$1 \leq \tau([x]) \tau([\cF(x)]).$$
\end{prop}

\begin{proof}
Compute
\begin{align}
\|\cF(x)\|_\infty &\leq \|x\|_1 \\
&= \tau(|x| [x]) \\
&\leq \|x\|_2 \|[x]\|_2 \\
&= \|\cF(x)\|_2 \|[x]\|_2 \\
&= \tau([\cF(x)]\cF(x)^*\cF(x)[\cF(x)])^{1/2} \|[x]\|_2 \\
&\leq \|\cF(x)\|_\infty \tau([\cF(x)])^{1/2} \|[x]\|_2 \\
&= \|\cF(x)\|_\infty \tau([\cF(x)])^{1/2} \tau([x])^{1/2}
\end{align}
where we used Proposition \ref{prop:RiemannLebesgue}, the Cauchy--Schwarz inequality, Proposition \ref{prop:Parseval} and the positivity of $\Tr$. If $x\neq 0$, \ie if $\cF(x)\neq 0$, dividing by $\|\cF(x)\|_\infty$ we get the statement.
\end{proof}

In our case at hand, $\tau([x]) = E([x])$ by Lemma \ref{lem:Estate}, hence also $\tau([x]) = \mu_E([x])$, and 
$$\tau([\cF(x)]) = \sum_\rho d(\rho) \rank((\cF(x))(\rho))$$
by our choice of normalization of $\Tr$. As a consequence, we obtain the following stronger version of the uncertainty principle \eqref{eq:DonohoStarkuncertainty}:

\begin{cor}[Donoho--Stark uncertainty principle]\label{prop:strongDonohoStarkuncertainty}
For every $x \in L^\infty(K)$,
$x\neq 0$, it holds
$$1 \leq \mu_E([x]) \sum_{\rho\in\supp \cF(x)} d(\rho) \rank((\cF(x))(\rho)).$$
\end{cor}

\bigskip
\noindent
{\bf Acknowledgements.}
We thank Dietmar Bisch for many fruitful discussions on the topics treated in this paper, Sebastiano Carpi for comments on a previous version of the manuscript, Corey Jones for hinting to us a proof idea concerning intermediate inclusions. Some of the results contained in this work have been presented at the conference \lq\lq International Workshop on Operator Theory and its Applications IWOTA Lancaster UK 2021\rq\rq\ during the special session \lq\lq Quantum Groups and Algebraic Quantum Field Theory\rq\rq, supported by EPSRC grant EP/T007524/1. We acknowledge support from the \lq\lq MIUR Excellence Department Project awarded to the Department of Mathematics, University of Rome Tor Vergata, CUP E83C18000100006\rq\rq.

\bigskip
\def\cprime{$'$}\newcommand{\noopsort}[1]{}

\bigskip

\address


\begin{bibdiv}
\begin{biblist}

\bib{AlAm2017}{article}{
      author={Alaghmandan, Mahmood},
      author={Amini, Massoud},
       title={Dual space and hyperdimension of compact hypergroups},
        date={2017},
        ISSN={0017-0895},
     journal={Glasg. Math. J.},
      volume={59},
      number={2},
       pages={421\ndash 435},
         url={https://mathscinet.ams.org/mathscinet-getitem?mr=3628938},
}

\bib{AcCe1982}{article}{
      author={Accardi, Luigi},
      author={Cecchini, Carlo},
       title={Conditional expectations in von {N}eumann algebras and a theorem
  of {T}akesaki},
        date={1982},
        ISSN={0022-1236},
     journal={J. Funct. Anal.},
      volume={45},
      number={2},
       pages={245\ndash 273},
         url={http://dx.doi.org/10.1016/0022-1236(82)90022-2},
}

\bib{AmMe2014}{article}{
      author={Amini, Massoud},
      author={Medghalchi, Ali~Reza},
       title={Amenability of compact hypergroup algebras},
        date={2014},
        ISSN={0025-584X},
     journal={Math. Nachr.},
      volume={287},
      number={14-15},
       pages={1609\ndash 1617},
         url={https://doi.org/10.1002/mana.201200284},
}

\bib{AlRu2008}{article}{
      author={Alagic, Gorjan},
      author={Russell, Alexander},
       title={Uncertainty principles for compact groups},
        date={2008},
        ISSN={0019-2082},
     journal={Illinois J. Math.},
      volume={52},
      number={4},
       pages={1315\ndash 1324},
         url={https://mathscinet.ams.org/mathscinet-getitem?mr=2595770},
}

\bib{Ar1969}{article}{
      author={Arveson, William~B.},
       title={Subalgebras of {$C^{\ast}$}-algebras},
        date={1969},
        ISSN={0001-5962},
     journal={Acta Math.},
      volume={123},
       pages={141\ndash 224},
         url={https://doi.org/10.1007/BF02392388},
}

\bib{BiChEvGiPe2020}{article}{
      author={Bischoff, Marcel},
      author={Charlesworth, Ian},
      author={Evington, Samuel},
      author={Giorgetti, Luca},
      author={Penneys, David},
       title={Distortion for multifactor bimodules and representations of
  multifusion categories},
        date={2020},
      note={\href{https://arxiv.org/abs/2010.01067}{Preprint arXiv:2010.01067}}
         url={https://arxiv.org/pdf/2010.01067.pdf},
}

\bib{BiDeGi2021}{article}{
      author={Bischoff, Marcel},
      author={Del~Vecchio, Simone},
      author={Giorgetti, Luca},
       title={Compact hypergroups from discrete subfactors},
        date={2021},
        ISSN={0022-1236},
     journal={J. Funct. Anal.},
      volume={281},
      number={1},
       pages={109004},
         url={https://mathscinet.ams.org/mathscinet-getitem?mr=4234861},
}

\bib{BcEv1998-I}{article}{
      author={B{\"o}ckenhauer, Jens},
      author={Evans, David~E.},
       title={{Modular invariants, graphs and {$\alpha$}-induction for nets of
  subfactors. {I}}},
        date={1998},
        ISSN={0010-3616},
     journal={Comm. Math. Phys.},
      volume={197},
      number={2},
       pages={361\ndash 386},
         url={http://dx.doi.org/10.1007/s002200050455},
}

\bib{BcEv1999-II}{article}{
      author={B{\"o}ckenhauer, J.},
      author={Evans, D.~E.},
       title={{Modular invariants, graphs and {$\alpha$}-induction for nets of
  subfactors. {II}}},
        date={1999},
        ISSN={0010-3616},
     journal={Comm. Math. Phys.},
      volume={200},
      number={1},
       pages={57\ndash 103},
         url={http://dx.doi.org/10.1007/s002200050523},
}

\bib{BcEv1999-III}{article}{
      author={B{\"o}ckenhauer, Jens},
      author={Evans, David~E.},
       title={{Modular invariants, graphs and {$\alpha$}-induction for nets of
  subfactors. {III}}},
        date={1999},
        ISSN={0010-3616},
     journal={Comm. Math. Phys.},
      volume={205},
      number={1},
       pages={183\ndash 228},
         url={http://dx.doi.org/10.1007/s002200050673},
}

\bib{BcEvKa2000}{article}{
      author={B{\"o}ckenhauer, Jens},
      author={Evans, David~E.},
      author={Kawahigashi, Yasuyuki},
       title={{Chiral structure of modular invariants for subfactors}},
        date={2000},
        ISSN={0010-3616},
     journal={Comm. Math. Phys.},
      volume={210},
      number={3},
       pages={733\ndash 784},
         url={http://dx.doi.org/10.1007/s002200050798},
}

\bib{BcEvKa1999}{article}{
      author={B{\"o}ckenhauer, Jens},
      author={Evans, David~E.},
      author={Kawahigashi, Yasuyuki},
       title={{On {$\alpha$}-induction, chiral generators and modular
  invariants for subfactors}},
        date={1999},
        ISSN={0010-3616},
     journal={Comm. Math. Phys.},
      volume={208},
      number={2},
       pages={429\ndash 487},
         url={http://dx.doi.org/10.1007/s002200050765},
}

\bib{BlHe1995}{book}{
      author={Bloom, Walter~R.},
      author={Heyer, Herbert},
       title={Harmonic analysis of probability measures on hypergroups},
      series={de Gruyter Studies in Mathematics},
   publisher={Walter de Gruyter \& Co., Berlin},
        date={1995},
      volume={20},
        ISBN={3-11-012105-0},
         url={http://dx.doi.org/10.1515/9783110877595},
}

\bib{Bi2016}{article}{
      author={Bischoff, Marcel},
       title={Generalized orbifold construction for conformal nets},
        date={2017},
        ISSN={0129-055X},
     journal={Rev. Math. Phys.},
      volume={29},
      number={1},
       pages={1750002, 53},
         url={http://dx.doi.org/10.1142/S0129055X17500027},
}

\bib{Bi1997}{incollection}{
      author={Bisch, Dietmar},
       title={Bimodules, higher relative commutants and the fusion algebra
  associated to a subfactor},
        date={1997},
   booktitle={Operator algebras and their applications ({W}aterloo, {ON},
  1994/1995)},
      series={Fields Inst. Commun.},
      volume={13},
   publisher={Amer. Math. Soc., Providence, RI},
       pages={13\ndash 63},
}

\bib{BiJo2000}{article}{
      author={Bisch, Dietmar},
      author={Jones, Vaughan},
       title={Singly generated planar algebras of small dimension},
        date={2000},
        ISSN={0012-7094},
     journal={Duke Math. J.},
      volume={101},
      number={1},
       pages={41\ndash 75},
         url={https://mathscinet.ams.org/mathscinet-getitem?mr=1733737},
}

\bib{BeKa1998}{book}{
      author={Berezansky, Yu.~M.},
      author={Kalyuzhnyi, A.~A.},
       title={Harmonic analysis in hypercomplex systems},
      series={Mathematics and its Applications},
   publisher={Kluwer Academic Publishers, Dordrecht},
        date={1998},
      volume={434},
        ISBN={0-7923-5029-4},
         url={https://doi.org/10.1007/978-94-017-1758-8},
        note={Translated from the 1992 Russian original by P. V. Malyshev and
  revised by the authors},
}

\bib{BiKaLoRe2014-2}{book}{
      author={Bischoff, Marcel},
      author={Kawahigashi, Yasuyuki},
      author={Longo, Roberto},
      author={Rehren, Karl-Henning},
       title={Tensor categories and endomorphisms of von {N}eumann
  algebras---with applications to quantum field theory},
      series={Springer Briefs in Mathematical Physics},
   publisher={Springer, Cham},
        date={2015},
      volume={3},
        ISBN={978-3-319-14300-2; 978-3-319-14301-9},
         url={http://dx.doi.org/10.1007/978-3-319-14301-9},
}

\bib{BiKaLoRe2014}{article}{
      author={Bischoff, Marcel},
      author={Kawahigashi, Yasuyuki},
      author={Longo, Roberto},
      author={Rehren, Karl-Henning},
       title={Phase {B}oundaries in {A}lgebraic {C}onformal {QFT}},
        date={2016},
        ISSN={0010-3616},
     journal={Comm. Math. Phys.},
      volume={342},
      number={1},
       pages={1\ndash 45},
         url={http://dx.doi.org/10.1007/s00220-015-2560-0},
}

\bib{CaCo2001}{article}{
      author={Carpi, Sebastiano},
      author={Conti, Roberto},
       title={Classification of subsystems for local nets with trivial
  superselection structure},
        date={2001},
        ISSN={0010-3616},
     journal={Comm. Math. Phys.},
      volume={217},
      number={1},
       pages={89\ndash 106},
         url={http://dx.doi.org/10.1007/PL00005550},
}

\bib{CaCo2001Siena}{incollection}{
      author={Carpi, Sebastiano},
      author={Conti, Roberto},
       title={Classification of subsystems, local symmetry generators and
  intrinsic definition of local observables},
        date={2001},
   booktitle={Mathematical physics in mathematics and physics ({S}iena, 2000)},
      series={Fields Inst. Commun.},
      volume={30},
   publisher={Amer. Math. Soc., Providence, RI},
       pages={83\ndash 103},
         url={https://doi.org/10.1007/pl00005550},
}

\bib{CaCo2005}{article}{
      author={Carpi, Sebastiano},
      author={Conti, Roberto},
       title={Classification of subsystems for graded-local nets with trivial
  superselection structure},
        date={2005},
        ISSN={0010-3616},
     journal={Comm. Math. Phys.},
      volume={253},
      number={2},
       pages={423\ndash 449},
         url={http://dx.doi.org/10.1007/s00220-004-1135-2},
}

\bib{CoDoRo2001}{article}{
      author={Conti, Roberto},
      author={Doplicher, Sergio},
      author={Roberts, John~E.},
       title={Superselection theory for subsystems},
        date={2001},
        ISSN={0010-3616},
     journal={Comm. Math. Phys.},
      volume={218},
      number={2},
       pages={263\ndash 281},
         url={http://dx.doi.org/10.1007/s002200100392},
}

\bib{CaHiKaLoXu15}{article}{
      author={Carpi, Sebastiano},
      author={Hillier, Robin},
      author={Kawahigashi, Yasuyuki},
      author={Longo, Roberto},
      author={Xu, Feng},
       title={{$N=2$} superconformal nets},
        date={2015},
        ISSN={0010-3616},
     journal={Comm. Math. Phys.},
      volume={336},
      number={3},
       pages={1285\ndash 1328},
         url={https://mathscinet.ams.org/mathscinet-getitem?mr=3324145},
}

\bib{CaKaLo2010}{article}{
      author={Carpi, Sebastiano},
      author={Kawahigashi, Yasuyuki},
      author={Longo, Roberto},
       title={On the {J}ones index values for conformal subnets},
        date={2010},
        ISSN={0377-9017},
     journal={Lett. Math. Phys.},
      volume={92},
      number={2},
       pages={99\ndash 108},
         url={https://doi.org/10.1007/s11005-010-0384-6},
}

\bib{ChNg2005}{article}{
      author={Chua, Kok~Seng},
      author={Ng, Wee~Seng},
       title={A simple proof of the uncertainty principle for compact groups},
        date={2005},
        ISSN={0723-0869},
     journal={Expo. Math.},
      volume={23},
      number={2},
       pages={147\ndash 150},
         url={https://mathscinet.ams.org/mathscinet-getitem?mr=2155006},
}

\bib{Co1980}{article}{
      author={Connes, A.},
       title={On the spatial theory of von {N}eumann algebras},
        date={1980},
        ISSN={0022-1236},
     journal={J. Functional Analysis},
      volume={35},
      number={2},
       pages={153\ndash 164},
         url={https://mathscinet.ams.org/mathscinet-getitem?mr=561983},
}

\bib{ChVa1999}{article}{
      author={Chapovsky, Yu.~A.},
      author={Vainerman, L.~I.},
       title={Compact quantum hypergroups},
        date={1999},
        ISSN={0379-4024},
     journal={J. Operator Theory},
      volume={41},
      number={2},
       pages={261\ndash 289},
}

\bib{Da1996}{article}{
      author={David, Marie-Claude},
       title={Paragroupe d'{A}drian {O}cneanu et alg\`ebre de {K}ac},
        date={1996},
        ISSN={0030-8730},
     journal={Pacific J. Math.},
      volume={172},
      number={2},
       pages={331\ndash 363},
         url={http://projecteuclid.org/euclid.pjm/1102366014},
}

\bib{DaGhGu2014}{article}{
      author={Das, Paramita},
      author={Ghosh, Shamindra~Kumar},
      author={Gupta, Ved~Prakash},
       title={Perturbations of planar algebras},
        date={2014},
        ISSN={0025-5521},
     journal={Math. Scand.},
      volume={114},
      number={1},
       pages={38\ndash 85},
         url={https://doi.org/10.7146/math.scand.a-16639},
}

\bib{DoHaRo1971}{article}{
      author={Doplicher, Sergio},
      author={Haag, Rudolf},
      author={Roberts, John~E.},
       title={Local observables and particle statistics. {I}},
        date={1971},
        ISSN={0010-3616},
     journal={Comm. Math. Phys.},
      volume={23},
       pages={199\ndash 230},
}

\bib{DiMa1971}{article}{
      author={Dixmier, J.},
      author={Mar\'{e}chal, O.},
       title={Vecteurs totalisateurs d'une alg{\`e}bre de von {N}eumann},
        date={1971},
        ISSN={0010-3616},
     journal={Comm. Math. Phys.},
      volume={22},
       pages={44\ndash 50},
         url={https://mathscinet.ams.org/mathscinet-getitem?mr=296708},
}

\bib{DoRo1972}{article}{
      author={Doplicher, Sergio},
      author={Roberts, John~E.},
       title={Fields, statistics and non-abelian gauge groups},
        date={1972},
        ISSN={0010-3616},
     journal={Comm. Math. Phys.},
      volume={28},
       pages={331\ndash 348},
}

\bib{De2013}{article}{
      author={Degenfeld-Schonburg, Sina},
       title={On the {H}ausdorff-{Y}oung theorem for commutative hypergroups},
        date={2013},
        ISSN={0010-1354},
     journal={Colloq. Math.},
      volume={131},
      number={2},
       pages={219\ndash 231},
         url={https://mathscinet.ams.org/mathscinet-getitem?mr=3092452},
}

\bib{DeFiRo2021}{article}{
      author={Del~Vecchio, Simone},
      author={Fidaleo, Francesco},
      author={Rossi, Stefano},
       title={Skew-product dynamical systems for crossed product
  {$C^*$}-algebras and their ergodic properties},
        date={2021},
        ISSN={0022-247X},
     journal={J. Math. Anal. Appl.},
      volume={503},
      number={1},
       pages={125302},
         url={https://doi.org/10.1016/j.jmaa.2021.125302},
}

\bib{DeGi2018}{article}{
      author={Del~Vecchio, Simone},
      author={Giorgetti, Luca},
       title={Infinite index extensions of local nets and defects},
        date={2018},
        ISSN={0129-055X},
     journal={Rev. Math. Phys.},
      volume={30},
      number={2},
       pages={1850002, 58},
         url={https://doi.org/10.1142/S0129055X18500022},
}

\bib{EGNO15}{book}{
      author={Etingof, Pavel},
      author={Gelaki, Shlomo},
      author={Nikshych, Dmitri},
      author={Ostrik, Victor},
       title={Tensor categories},
      series={Mathematical Surveys and Monographs},
   publisher={American Mathematical Society, Providence, RI},
        date={2015},
      volume={205},
        ISBN={978-1-4704-2024-6},
}

\bib{EvKa1998Book}{book}{
      author={Evans, David~E.},
      author={Kawahigashi, Yasuyuki},
       title={{Quantum symmetries on operator algebras}},
      series={{Oxford Mathematical Monographs}},
   publisher={The Clarendon Press Oxford University Press},
     address={New York},
        date={1998},
        ISBN={0-19-851175-2},
        note={Oxford Science Publications},
}

\bib{EnNe1996}{article}{
      author={Enock, Michel},
      author={Nest, Ryszard},
       title={Irreducible inclusions of factors, multiplicative unitaries, and
  {K}ac algebras},
        date={1996},
        ISSN={0022-1236},
     journal={J. Funct. Anal.},
      volume={137},
      number={2},
       pages={466\ndash 543},
         url={http://dx.doi.org/10.1006/jfan.1996.0053},
}

\bib{EvPi2003}{article}{
      author={Evans, David~E.},
      author={Pinto, Paulo~R.},
       title={{Subfactor realisation of modular invariants}},
        date={2003},
        ISSN={0010-3616},
     journal={Comm. Math. Phys.},
      volume={237},
      number={1-2},
       pages={309\ndash 363},
         url={http://dx.doi.org/10.1142/S0129167X12500309},
}

\bib{FiIs1999}{article}{
      author={Fidaleo, F.},
      author={Isola, T.},
       title={The canonical endomorphism for infinite index inclusions},
        date={1999},
        ISSN={0232-2064},
     journal={Z. Anal. Anwendungen},
      volume={18},
      number={1},
       pages={47\ndash 66},
         url={http://dx.doi.org/10.4171/ZAA/869},
}

\bib{FrReSc1989}{article}{
      author={Fredenhagen, K.},
      author={Rehren, K.-H.},
      author={Schroer, B.},
       title={{Superselection sectors with braid group statistics and exchange
  algebras. {I}.\ {G}eneral theory}},
        date={1989},
        ISSN={0010-3616},
     journal={Comm. Math. Phys.},
      volume={125},
      number={2},
       pages={201\ndash 226},
         url={http://projecteuclid.org/getRecord?id=euclid.cmp/1104179464},
}

\bib{GoHaJo1989}{book}{
      author={Goodman, Frederick~M.},
      author={de~la Harpe, Pierre},
      author={Jones, Vaughan F.~R.},
       title={{Coxeter graphs and towers of algebras}},
      series={{Mathematical Sciences Research Institute Publications}},
   publisher={Springer-Verlag},
     address={New York},
        date={1989},
      volume={14},
        ISBN={0-387-96979-9},
         url={http://dx.doi.org/10.1007/978-1-4613-9641-3},
}

\bib{GiLo2019}{article}{
      author={Giorgetti, Luca},
      author={Longo, Roberto},
       title={Minimal index and dimension for 2-{$C^*$}-categories with
  finite-dimensional centers},
        date={2019},
        ISSN={0010-3616},
     journal={Comm. Math. Phys.},
      volume={370},
      number={2},
       pages={719\ndash 757},
         url={https://doi.org/10.1007/s00220-018-3266-x},
}

\bib{GiRe2018}{article}{
      author={Giorgetti, Luca},
      author={Rehren, Karl-Henning},
       title={Braided categories of endomorphisms as invariants for local
  quantum field theories},
        date={2018},
        ISSN={0010-3616},
     journal={Comm. Math. Phys.},
      volume={357},
      number={1},
       pages={3\ndash 41},
         url={https://doi.org/10.1007/s00220-017-2937-3},
}

\bib{GiYu2019}{article}{
      author={Giorgetti, Luca},
      author={Yuan, Wei},
       title={Realization of rigid $C^*$-tensor categories via {T}omita
  bimodules.},
        date={2019},
        ISSN={0379-4024},
     journal={J. Operator Theory},
      volume={81},
      number={2},
       pages={433\ndash479},
         url={https://doi.org/10.7900/jot.2018mar08.2219},
}

\bib{Ha1979I}{article}{
      author={Haagerup, Uffe},
       title={Operator-valued weights in von {N}eumann algebras. {I}},
        date={1979},
        ISSN={0022-1236},
     journal={J. Funct. Anal.},
      volume={32},
      number={2},
       pages={175\ndash 206},
         url={http://dx.doi.org/10.1016/0022-1236(79)90053-3},
}

\bib{Ha1979II}{article}{
      author={Haagerup, Uffe},
       title={Operator-valued weights in von {N}eumann algebras. {II}},
        date={1979},
        ISSN={0022-1236},
     journal={J. Funct. Anal.},
      volume={33},
      number={3},
       pages={339\ndash 361},
         url={http://dx.doi.org/10.1016/0022-1236(79)90072-7},
}

\bib{Ha}{book}{
      author={Haag, Rudolf},
       title={{Local quantum physics}},
   publisher={Springer Berlin},
        date={1996},
}

\bib{HeOc1989}{article}{
      author={Herman, Richard~H.},
      author={Ocneanu, Adrian},
       title={Index theory and {G}alois theory for infinite index inclusions of
  factors},
        date={1989},
        ISSN={0764-4442},
     journal={C. R. Acad. Sci. Paris S{\'e}r. I Math.},
      volume={309},
      number={17},
       pages={923\ndash 927},
}

\bib{IzLoPo1998}{article}{
      author={Izumi, Masaki},
      author={Longo, Roberto},
      author={Popa, Sorin},
       title={{A {G}alois correspondence for compact groups of automorphisms of
  von {N}eumann algebras with a generalization to {K}ac algebras}},
        date={1998},
        ISSN={0022-1236},
     journal={J. Funct. Anal.},
      volume={155},
      number={1},
       pages={25\ndash 63},
         url={http://dx.doi.org/10.1006/jfan.1997.3228},
}

\bib{JaJiLiReWu2020}{article}{
      author={Jaffe, Arthur},
      author={Jiang, Chunlan},
      author={Liu, Zhengwei},
      author={Ren, Yunxiang},
      author={Wu, Jinsong},
       title={Quantum {F}ourier analysis},
        date={2020},
        ISSN={0027-8424},
     journal={Proc. Natl. Acad. Sci. USA},
      volume={117},
      number={20},
       pages={10715\ndash 10720},
         url={https://mathscinet.ams.org/mathscinet-getitem?mr=4236188},
}

\bib{JiLiWu2016}{article}{
      author={Jiang, Chunlan},
      author={Liu, Zhengwei},
      author={Wu, Jinsong},
       title={Noncommutative uncertainty principles},
        date={2016},
        ISSN={0022-1236},
     journal={J. Funct. Anal.},
      volume={270},
      number={1},
       pages={264\ndash 311},
         url={https://doi.org/10.1016/j.jfa.2015.08.007},
}

\bib{JiLiWu2018}{article}{
      author={Jiang, Chunlan},
      author={Liu, Zhengwei},
      author={Wu, Jinsong},
       title={Uncertainty principles for locally compact quantum groups},
        date={2018},
        ISSN={0022-1236},
     journal={J. Funct. Anal.},
      volume={274},
      number={8},
       pages={2399\ndash 2445},
         url={https://mathscinet.ams.org/mathscinet-getitem?mr=3767437},
}

\bib{Jo1983}{article}{
      author={Jones, V. F.~R.},
       title={{Index for subfactors}},
        date={1983},
        ISSN={0020-9910},
     journal={Invent. Math.},
      volume={72},
      number={1},
       pages={1\ndash 25},
         url={http://dx.doi.org/10.1007/BF01389127},
}

\bib{Jo1999}{article}{
      author={Jones, V. F.~R.},
       title={{Planar Algebras, $I$}},
        date={1999},
        note={\href{https://arxiv.org/abs/math/9909027}{Preprint arXiv:math/9909027}}
         url={https://arxiv.org/abs/math/9909027},
}

\bib{JoPe2011}{article}{
      author={Jones, Vaughan F.~R.},
      author={Penneys, David},
       title={The embedding theorem for finite depth subfactor planar
  algebras},
        date={2011},
        ISSN={1663-487X},
     journal={Quantum Topol.},
      volume={2},
      number={3},
       pages={301\ndash 337},
         url={https://mathscinet.ams.org/mathscinet-getitem?mr=2812459},
}

\bib{JoPe2019}{article}{
      author={Jones, Corey},
      author={Penneys, David},
       title={Realizations of algebra objects and discrete subfactors},
        date={2019},
        ISSN={0001-8708},
     journal={Adv. Math.},
      volume={350},
       pages={588\ndash 661},
         url={https://doi.org/10.1016/j.aim.2019.04.039},
}

\bib{JoSu1997}{book}{
      author={Jones, V. F.~R.},
      author={Sunder, V.~S.},
       title={{Introduction to subfactors}},
      series={{London Mathematical Society Lecture Note Series}},
   publisher={Cambridge University Press},
     address={Cambridge},
        date={1997},
      volume={234},
        ISBN={0-521-58420-5},
         url={http://dx.doi.org/10.1017/CBO9780511566219},
}

\bib{Palermo1989}{proceedings}{
      editor={Kastler, Daniel},
       title={The algebraic theory of superselection sectors},
   publisher={World Scientific Publishing Co., Inc., River Edge, NJ},
        date={1990},
        ISBN={981-02-0206-7},
         url={https://doi.org/10.1142/1093},
}

\bib{Ko1984}{article}{
      author={Kosaki, Hideki},
       title={Applications of the complex interpolation method to a von
  {N}eumann algebra: noncommutative {$L^{p}$}-spaces},
        date={1984},
        ISSN={0022-1236},
     journal={J. Funct. Anal.},
      volume={56},
      number={1},
       pages={29\ndash 78},
         url={https://mathscinet.ams.org/mathscinet-getitem?mr=735704},
}

\bib{Ko1986}{article}{
      author={Kosaki, Hideki},
       title={{Extension of {J}ones' theory on index to arbitrary factors}},
        date={1986},
        ISSN={0022-1236},
     journal={J. Funct. Anal.},
      volume={66},
      number={1},
       pages={123\ndash 140},
         url={http://dx.doi.org/10.1016/0022-1236(86)90085-6},
}

\bib{Ko1989}{article}{
      author={Kosaki, Hideki},
       title={Characterization of crossed product (properly infinite case)},
        date={1989},
        ISSN={0030-8730},
     journal={Pacific J. Math.},
      volume={137},
      number={1},
       pages={159\ndash 167},
         url={http://projecteuclid.org/euclid.pjm/1102650542},
}

\bib{Ko1998}{book}{
      author={Kosaki, Hideki},
       title={{Type {III} factors and index theory}},
      series={{Lecture Notes Series}},
   publisher={Seoul National University Research Institute of Mathematics
  Global Analysis Research Center},
     address={Seoul},
        date={1998},
      volume={43},
}

\bib{KaPoCh2010}{article}{
      author={Kalyuzhnyi, A.~A.},
      author={Podkolzin, G.~B.},
      author={Chapovsky, Yu.~A.},
       title={Harmonic analysis on a locally compact hypergroup},
        date={2010},
        ISSN={1029-3531},
     journal={Methods Funct. Anal. Topology},
      volume={16},
      number={4},
       pages={304\ndash 332},
}

\bib{KuRu2020-arxiv}{article}{
      author={Kumar, Vishvesh},
      author={Ruzhansky, Michael},
       title={{H}ardy-{L}ittlewood inequality and ${L}^p$-${L}^q$ {F}ourier
  multipliers on compact hypergroups},
        date={2020},
      note={\href{https://arxiv.org/abs/2005.08464}{Preprint arXiv:2005.08464}}
         url={https://arxiv.org/pdf/2005.08464.pdf},
}

\bib{KuSa2020}{article}{
      author={Kumar, Vishvesh},
      author={Sarma, Ritumoni},
       title={The {H}ausdorff-{Y}oung inequality for {O}rlicz spaces on compact
  hypergroups},
        date={2020},
        ISSN={0010-1354},
     journal={Colloq. Math.},
      volume={160},
      number={1},
       pages={41\ndash 51},
         url={https://mathscinet.ams.org/mathscinet-getitem?mr=4071812},
}

\bib{KuVa2003}{article}{
      author={Kustermans, Johan},
      author={Vaes, Stefaan},
       title={Locally compact quantum groups in the von {N}eumann algebraic
  setting},
        date={2003},
        ISSN={0025-5521},
     journal={Math. Scand.},
      volume={92},
      number={1},
       pages={68\ndash 92},
         url={https://doi.org/10.7146/math.scand.a-14394},
}

\bib{Lo2018}{article}{
      author={Longo, Roberto},
       title={On {L}andauer's principle and bound for infinite systems},
        date={2018},
        ISSN={0010-3616},
     journal={Comm. Math. Phys.},
      volume={363},
      number={2},
       pages={531\ndash 560},
  url={https://doi-org.proxy.library.ohio.edu/10.1007/s00220-018-3116-x},
}

\bib{Lo1987}{article}{
      author={Longo, Roberto},
       title={{Simple injective subfactors}},
        date={1987},
        ISSN={0001-8708},
     journal={Adv. Math.},
      volume={63},
      number={2},
       pages={152\ndash 171},
  url={http://www.sciencedirect.com/science/article/B6W9F-4CRY32C-NJ/2/ad65ab20fcbb926ca7c0c1d780ddc8fc},
}

\bib{Lo1989}{article}{
      author={Longo, Roberto},
       title={{Index of subfactors and statistics of quantum fields. I}},
        date={1989},
     journal={Comm. Math. Phys.},
      volume={126},
       pages={217\ndash 247},
}

\bib{Lo1990}{article}{
      author={Longo, Roberto},
       title={{Index of subfactors and statistics of quantum fields. II.
  Correspondences, Braid Group Statistics and Jones Polynomial}},
        date={1990},
     journal={Comm. Math. Phys.},
      volume={130},
       pages={285\ndash 309},
}

\bib{Lo1992}{article}{
      author={Longo, Roberto},
       title={Minimal index and braided subfactors},
        date={1992},
        ISSN={0022-1236},
     journal={J. Funct. Anal.},
      volume={109},
      number={1},
       pages={98\ndash 112},
         url={http://dx.doi.org/10.1016/0022-1236(92)90013-9},
}

\bib{Lo1994}{article}{
      author={Longo, Roberto},
       title={{A duality for {H}opf algebras and for subfactors. {I}}},
        date={1994},
        ISSN={0010-3616},
     journal={Comm. Math. Phys.},
      volume={159},
      number={1},
       pages={133\ndash 150},
         url={http://projecteuclid.org/getRecord?id=euclid.cmp/1104254494},
}

\bib{LiPaWu2019-arxiv}{article}{
      author={Liu, Zhengwei},
      author={Palcoux, Sebastien},
      author={Wu, Jinsong},
       title={Fusion bialgebras and fourier analysis},
        date={2019},
      note={\href{https://arxiv.org/abs/1910.12059}{Preprint arXiv:1910.12059}}
         url={https://arxiv.org/pdf/1910.12059.pdf},
}

\bib{LoRe1995}{article}{
      author={Longo, Roberto},
      author={Rehren, Karl-Henning},
       title={{Nets of Subfactors}},
        date={1995},
     journal={Rev. Math. Phys.},
      volume={7},
       pages={567\ndash 597},
}

\bib{LoRo1997}{article}{
      author={Longo, R.},
      author={Roberts, J.~E.},
       title={{A theory of dimension}},
        date={1997},
        ISSN={0920-3036},
     journal={K-Theory},
      volume={11},
      number={2},
       pages={103\ndash 159},
      note={Available at \href{http://dx.doi.org/10.1023/A:1007714415067}{http://dx.doi.org/10.1023/A:1007714415067}}
         url={http://dx.doi.org/10.1023/A:1007714415067},
}

\bib{LiWu2017}{article}{
      author={Liu, Zhengwei},
      author={Wu, Jinsong},
       title={Uncertainty principles for {K}ac algebras},
        date={2017},
        ISSN={0022-2488},
     journal={J. Math. Phys.},
      volume={58},
      number={5},
       pages={052102, 12},
         url={https://mathscinet.ams.org/mathscinet-getitem?mr=3652772},
}

\bib{Ne1974}{article}{
      author={Nelson, Edward},
       title={Notes on non-commutative integration},
        date={1974},
     journal={J. Funct. Anal.},
      volume={15},
       pages={103\ndash 116},
         url={https://mathscinet.ams.org/mathscinet-getitem?mr=0355628},
}

\bib{NiStZs2003}{article}{
      author={Niculescu, Constantin~P.},
      author={Str{\"o}h, Anton},
      author={Zsid{\'o}, L{\'a}szl{\'o}},
       title={Noncommutative extensions of classical and multiple recurrence
  theorems},
        date={2003},
        ISSN={0379-4024},
     journal={J. Operator Theory},
      volume={50},
      number={1},
       pages={3\ndash 52},
}

\bib{NaTa1960}{article}{
      author={Nakamura, Masahiro},
      author={Takeda, Zir\^{o}},
       title={A {G}alois theory for finite factors},
        date={1960},
        ISSN={0021-4280},
     journal={Proc. Japan Acad.},
      volume={36},
       pages={258\ndash 260},
         url={https://mathscinet.ams.org/mathscinet-getitem?mr=123925},
}

\bib{NiWi1995}{article}{
      author={Nill, Florian},
      author={Wiesbrock, Hans-Werner},
       title={A comment on {J}ones inclusions with infinite index},
        date={1995},
        ISSN={0129-055X},
     journal={Rev. Math. Phys.},
      volume={7},
      number={4},
       pages={599\ndash 630},
         url={http://dx.doi.org/10.1142/S0129055X95000244},
}

\bib{Oc1988}{incollection}{
      author={Ocneanu, Adrian},
       title={Quantized groups, string algebras and {G}alois theory for
  algebras},
        date={1988},
   booktitle={Operator algebras and applications, {V}ol.\ 2},
      series={London Math. Soc. Lecture Note Ser.},
      volume={136},
   publisher={Cambridge Univ. Press, Cambridge},
       pages={119\ndash 172},
}

\bib{Oc1991}{article}{
      author={Ocneanu, Adrian},
       title={{Quantum symmetry, differential geometry of finite graphs and
  classification of subfactors}},
        date={1991},
     journal={University of Tokyo Seminary Notes},
      number={45},
        note={(Notes recorded by Y. Kawahigashi)},
}

\bib{Pa1973}{article}{
      author={Paschke, William~L.},
       title={Inner product modules over {$B^{\ast} $}-algebras},
        date={1973},
        ISSN={0002-9947},
     journal={Trans. Amer. Math. Soc.},
      volume={182},
       pages={443\ndash 468},
}

\bib{Po1990}{article}{
      author={Popa, S.},
       title={Classification of subfactors: the reduction to commuting
  squares},
        date={1990},
        ISSN={0020-9910},
     journal={Invent. Math.},
      volume={101},
      number={1},
       pages={19\ndash 43},
         url={https://doi.org/10.1007/BF01231494},
}

\bib{Po1995}{article}{
      author={Popa, Sorin},
       title={An axiomatization of the lattice of higher relative commutants of
  a subfactor},
        date={1995},
        ISSN={0020-9910},
     journal={Invent. Math.},
      volume={120},
      number={3},
       pages={427\ndash 445},
         url={http://dx.doi.org/10.1007/BF01241137},
}

\bib{Po1995a}{book}{
      author={Popa, Sorin},
       title={Classification of subfactors and their endomorphisms},
      series={CBMS Regional Conference Series in Mathematics},
   publisher={Published for the Conference Board of the Mathematical Sciences,
  Washington, DC; by the American Mathematical Society, Providence, RI},
        date={1995},
      volume={86},
        ISBN={0-8218-0321-2},
}

\bib{Po1999}{article}{
      author={Popa, Sorin},
       title={Some properties of the symmetric enveloping algebra of a
  subfactor, with applications to amenability and property {T}},
        date={1999},
        ISSN={1431-0635},
     journal={Doc. Math.},
      volume={4},
       pages={665\ndash 744},
}

\bib{PiPo1986}{article}{
      author={Pimsner, M.},
      author={Popa, S.},
       title={{Entropy and index for subfactors}},
        date={1986},
     journal={Ann. Sci. Ecole Norm. Sup},
      volume={19},
      number={4},
       pages={57\ndash 106},
}

\bib{PoShVa2018}{article}{
      author={Popa, Sorin},
      author={Shlyakhtenko, Dimitri},
      author={Vaes, Stefaan},
       title={Cohomology and {$L^2$}-{B}etti numbers for subfactors and
  quasi-regular inclusions},
        date={2018},
        ISSN={1073-7928},
     journal={Int. Math. Res. Not. IMRN},
      number={8},
       pages={2241\ndash 2331},
         url={https://doi.org/10.1093/imrn/rnw304},
}

\bib{Re1995}{article}{
      author={Rehren, Karl-Henning},
       title={On the range of the index of subfactors},
        date={1995},
        ISSN={0022-1236},
     journal={J. Funct. Anal.},
      volume={134},
      number={1},
       pages={183\ndash 193},
         url={https://doi.org/10.1006/jfan.1995.1141},
}

\bib{Ro1976}{article}{
      author={Roberts, John~E.},
       title={{Local cohomology and superselection structure}},
        date={1976},
        ISSN={0010-3616},
     journal={Comm. Math. Phys.},
      volume={51},
      number={2},
       pages={107\ndash 119},
}

\bib{Ro1990}{article}{
      author={Roberts, John~E.},
       title={{Lectures on algebraic quantum field theory}},
       pages={1\ndash 112},
        note={In {\protect\NoHyper\cite{Palermo1989}\protect\endNoHyper}},
}

\bib{RudRC}{book}{
      author={Rudin, Walter},
       title={Real and complex analysis},
     edition={Second},
   publisher={McGraw-Hill Book Co., New York-D\"{u}sseldorf-Johannesburg},
        date={1974},
        note={McGraw-Hill Series in Higher Mathematics},
}

\bib{Sa1997}{article}{
      author={Sato, Nobuya},
       title={Fourier transform for paragroups and its application to the depth
  two case},
        date={1997},
        ISSN={0034-5318},
     journal={Publ. Res. Inst. Math. Sci.},
      volume={33},
      number={2},
       pages={189\ndash 222},
         url={http://dx.doi.org/10.2977/prims/1195145447},
}

\bib{Sa1983}{article}{
      author={Sauvageot, Jean-Luc},
       title={Sur le produit tensoriel relatif d'espaces de {H}ilbert},
        date={1983},
     journal={J. Operator Theory},
      volume={9},
      number={2},
       pages={237\ndash 252},
}

\bib{St1997}{incollection}{
      author={St{\o}rmer, E.},
       title={Conditional expectations and projection maps of von {N}eumann
  algebras},
        date={1997},
   booktitle={Operator algebras and applications ({S}amos, 1996)},
      series={NATO Adv. Sci. Inst. Ser. C Math. Phys. Sci.},
      volume={495},
   publisher={Kluwer Acad. Publ., Dordrecht},
       pages={449\ndash 461},
         url={https://mathscinet.ams.org/mathscinet-getitem?mr=1462691},
}

\bib{SuWi2003}{article}{
      author={Sunder, V.~S.},
      author={Wildberger, N.~J.},
       title={Actions of finite hypergroups},
        date={2003},
        ISSN={0925-9899},
     journal={J. Algebraic Combin.},
      volume={18},
      number={2},
       pages={135\ndash 151},
         url={http://dx.doi.org/10.1023/A:1025107014451},
}

\bib{Sz1994}{article}{
      author={Szyma\'nski, Wojciech},
       title={Finite index subfactors and {H}opf algebra crossed products},
        date={1994},
        ISSN={0002-9939},
     journal={Proc. Amer. Math. Soc.},
      volume={120},
      number={2},
       pages={519\ndash 528},
         url={http://dx.doi.org/10.2307/2159890},
}

\bib{TePhD}{thesis}{
      author={Terp, Marianne},
       title={{$L^p$} spaces associated with von {N}eumann algebras.},
        type={Ph.D. Thesis},
        date={1981},
}

\bib{Te1982}{article}{
      author={Terp, Marianne},
       title={Interpolation spaces between a von {N}eumann algebra and its
  predual},
        date={1982},
        ISSN={0379-4024},
     journal={J. Operator Theory},
      volume={8},
      number={2},
       pages={327\ndash 360},
         url={https://mathscinet.ams.org/mathscinet-getitem?mr=677418},
}

\bib{To2009}{article}{
      author={Tomatsu, Reiji},
       title={A {G}alois correspondence for compact quantum group actions},
        date={2009},
        ISSN={0075-4102},
     journal={J. Reine Angew. Math.},
      volume={633},
       pages={165\ndash 182},
         url={https://doi.org/10.1515/CRELLE.2009.063},
}

\bib{Va2001}{article}{
      author={Vaes, Stefaan},
       title={The unitary implementation of a locally compact quantum group
  action},
        date={2001},
        ISSN={0022-1236},
     journal={J. Funct. Anal.},
      volume={180},
      number={2},
       pages={426\ndash 480},
         url={https://doi.org/10.1006/jfan.2000.3704},
}

\bib{Vr1979}{article}{
      author={Vrem, Richard~C.},
       title={Harmonic analysis on compact hypergroups},
        date={1979},
        ISSN={0030-8730},
     journal={Pacific J. Math.},
      volume={85},
      number={1},
       pages={239\ndash 251},
         url={http://projecteuclid.org/euclid.pjm/1102784093},
}

\bib{Xu2005}{article}{
      author={Xu, Feng},
       title={{Strong additivity and conformal nets}},
        date={2005},
        ISSN={0030-8730},
     journal={Pacific J. Math.},
      volume={221},
      number={1},
       pages={167\ndash 199},
         url={http://dx.doi.org/10.2140/pjm.2005.221.167},
}

\bib{Xu1998}{article}{
      author={Xu, Feng},
       title={New braided endomorphisms from conformal inclusions},
        date={1998},
        ISSN={0010-3616},
     journal={Comm. Math. Phys.},
      volume={192},
      number={2},
       pages={349\ndash 403},
         url={http://dx.doi.org/10.1007/s002200050302},
}

\end{biblist}
\end{bibdiv}
\end{document}